\documentclass[11pt]{amsart}
\usepackage[utf8]{inputenc}
\usepackage{amsmath}
\usepackage{amssymb}
\usepackage{amsthm}
\usepackage{bbm}
\usepackage{fullpage}
\usepackage{mathrsfs}
\usepackage{tikz-cd}
\usepackage{tabu}
\usepackage{arydshln}
\usepackage{graphicx}
\usepackage[all]{xy}
\usepackage{pdflscape}
\usepackage{multicol}
\usepackage{multirow}
\usepackage{mathtools}
\usepackage{hyperref}
\usepackage{enumitem}
\usepackage[toc,page]{appendix}
\usepackage[english]{babel}
\usepackage{setspace}
\usepackage{marvosym}
\usepackage{makecell}
\allowdisplaybreaks[1]

\newcommand{\fixIt}[1]{{\color{black} #1}}

\DeclareMathOperator{\Ima}{Im}
\DeclareMathOperator{\Spn}{Span}
\DeclareMathOperator{\Hom}{Hom}
\DeclareMathOperator{\End}{End}

\DeclareMathOperator{\Ind}{Ind}
\DeclareMathOperator{\ad}{ad}
\DeclareMathOperator{\Ad}{Ad}
\DeclareMathOperator{\mult}{\mathit{m}}
\DeclareMathOperator{\sdim}{sdim}

\DeclareMathOperator{\str}{str}

\DeclareMathOperator{\diag}{diag}
\DeclareMathOperator{\Res}{\mathtt{Res}}
\DeclareMathOperator{\std}{st}
\DeclareMathOperator{\opp}{op}

\newcommand{\sltwo}{\mathfrak{sl}(2)}
\newcommand{\supp}[1]{\mathrm{Supp}\left(#1\right)}
\newcommand{\gl}{\mathfrak{gl}}

\newcommand{\Uk}{\mathfrak{U}}

\newcommand{\Zk}{\mathfrak{Z}}
\newcommand{\gk}{\mathfrak{g}}
\newcommand{\kk}{\mathfrak{k}}

\newcommand{\tk}{\mathfrak{t}}
\newcommand{\hk}{\mathfrak{h}}
\newcommand{\Hk}{\mathfrak{H}}
\newcommand{\bk}{\mathfrak{b}}
\newcommand{\Bk}{\mathfrak{B}}
\newcommand{\ak}{\mathfrak{a}}
\newcommand{\pk}{\mathfrak{p}}
\newcommand{\nk}{\mathfrak{n}}
\newcommand{\Nk}{\mathfrak{N}}
\newcommand{\Sk}{\mathfrak{S}}
\newcommand{\Pk}{\mathfrak{P}}
\newcommand{\qk}{\mathfrak{q}}
\newcommand{\Lk}{\mathfrak{L}}
\newcommand{\D}{\mathfrak{D}}
\newcommand{\C}{\mathbb{C}}

\newcommand{\Z}{\mathbb{Z}}
\newcommand{\quoUk}{\Uk^\kk/(\Uk\kk)^\kk}
\newcommand{\pplus}{\mathfrak{p}^+}
\newcommand{\pminus}{\mathfrak{p}^-}
\newcommand{\ev}{_{\overline{0}}}
\newcommand{\od}{_{\overline{1}}}

\newcommand{\epu}{\epsilon}
\newcommand{\epr}{\alpha\bos}
\newcommand{\dau}{\delta}
\newcommand{\dar}{\alpha\fer}
\newcommand{\bos}{^{\textsc{B}}}
\newcommand{\fer}{^{\textsc{F}}}
\newcommand{\bnf}{^{\textsc{B/F}}}

\newcommand{\HCso}{\Sigma_\perp}
\newcommand{\SEvR}{\mathscr{S}^\mathtt{0}}
\newcommand{\fk}{\mathsf{k}}
\newcommand{\fh}{\mathsf{h}}
\newcommand{\fp}{\mathsf{p}}
\newcommand{\fq}{\mathsf{q}}
\newcommand{\fr}{\mathsf{r}}
\newcommand{\fs}{\mathsf{s}}
\newcommand{\cH}{\mathscr{H}}
\newcommand{\ii}{\mathsf{i}}

\newcommand{\sD}{\mathscr{D}}
\newcommand{\RingEv}{\Lambda^\mathtt{0}}
\newcommand{\hcHomoGK}{{\Gamma^\mathtt{0}}}
\newcommand{\hcIsoGK}{\gamma^\mathtt{0}}
\newcommand{\hcZIso}{\gamma}
\newcommand{\Wmu}{W_\mu}
\newcommand{\Wlambda}{W_\lambda}

\newcommand{\commentout}[1]{}

\newtheorem{main}{Theorem}

\theoremstyle{remark}

\theoremstyle{plain}
\newtheorem{thm}{Theorem}[section]
\newtheorem*{thm*}{Theorem}
\newtheorem{prop}[thm]{Proposition}
\newtheorem*{prop*}{Proposition}
\newtheorem{lem}[thm]{Lemma}
\newtheorem*{lem*}{Lemma}

\theoremstyle{definition}
\newtheorem{defn}[thm]{Definition}
\newtheorem{rmk}{Remark}
\newtheorem{cor}[thm]{Corollary}

\usepackage{fancyhdr}
\title{Eigenvalues of supersymmetric Shimura operators and interpolation polynomials}

\author{Siddhartha SAHI}
\address{Department of Mathematics, Rutgers University\\
Hill Center for the Mathematical Sciences\\
110 Frelinghuysen Rd.\\
Piscataway, NJ 08854-8019}
\email{sahi@math.rutgers.edu}

\author{Songhao ZHU}
\address{School of Mathematics, Georgia Institute of Technology\\
Atlanta, GA, 30332, USA}
\email{zhu.math@gatech.edu}

\thanks{The research of S. Sahi was partially supported by NSF grants DMS-1939600 and 2001537, by Simons Foundation grant 509766, and by the AIM SQuaRE program. }

\date{\today}

\begin{document}

\begin{abstract}
The Shimura operators are a certain distinguished basis for invariant differential operators on a Hermitian symmetric space. Answering a question of Shimura, Sahi and Zhang showed that the Harish-Chandra images of these operators are specializations of certain $BC$-symmetric interpolation polynomials that were defined by Okounkov. 

We consider the analogs of Shimura operators for the Hermitian symmetric superpair $(\gk,\kk)$ where $\gk= \gl(2p|2q)$ and $\kk=\gl(p|q)\oplus\gl(p|q)$ and we prove their Harish-Chandra images are specializations of certain $BC$-supersymmetric interpolation polynomials introduced by Sergeev and Veselov.
\end{abstract}


\maketitle



\section{Introduction}
The Harish-Chandra homomorphism for a symmetric space gives an explicit isomorphism of the algebra of invariant differential operators with a certain polynomial algebra, and knowing the Harish-Chandra image of an operator allows one to determine its spectrum. In \cite{Shimura90} Shimura introduced a basis for the algebra of invariant differential operators on a {\it Hermitian} symmetric space, and formulated the problem of determining their eigenvalues.  Shimura's problem was solved in \cite{SZ2019PoSo} where it was shown that the Harish-Chandra images of Shimura operators are specializations of certain interpolation polynomials of Type \textit{BC} introduced by Okounkov \cite{Okounkov98}.

Our main result, Theorem~\ref{thm:A} below, solves the analogous problem for the Hermitian symmetric superpair $(\gk, \kk)$ with $\gk = \gl(2p|2q)$ and $\kk = \gl(p|q) \oplus \gl(p|q)$. 

Let $\Uk$ be the universal enveloping algebra of $\gk$. 
The algebra of invariant differential operators is then given by the quotient of $\kk$-invariants $\D :=\quoUk$. In Section~\ref{subsec:superSO} we describe a basis $\sD_\mu$ of $\D$ indexed by the set $\cH := \cH(p, q)$ of $(p,q)$-hook partitions, which are partitions $\mu$ that satisfy $\mu_{p+1}\leq q$. 
These $\sD_\mu$ are the super analogs of Shimura operators, and their definition involves the Cartan decomposition $\gk =\kk \oplus \pk=\kk \oplus \pplus \oplus \pminus$, and the multiplicity-free $\kk$-decompositions of $\Sk(\pplus)$ and $\Sk(\pminus)$, which have $\kk$-summands $\Wmu$ and $\Wmu^*$, naturally indexed by $\mu\in\cH$  \cite{CW2000HDfL}.

The Harish-Chandra homomorphism (Section~\ref{sec:prelim}) is an algebra map  $\hcIsoGK:\D\to \Pk(\ak^*)$, where $\ak$ is an even Cartan subspace of $\pk$ and $\Pk(\ak^*)$ is the algebra of polynomials on $\ak^*$. For our pair $(\gk, \kk)$ the image of $\hcIsoGK$ can be identified with the ring $\SEvR_{p, q}$ of even supersymmetric polynomials in $p+q$ variables.
In Section~\ref{sec:BCPoly} we introduce a basis $\{J_\mu:\mu\in \cH\}$ of $\SEvR_{p, q}$. These $J_\mu$ are suitable specializations of the supersymmetric interpolation polynomials of Type \textit{BC} introduced by Sergeev and Veselov \cite{SV09}, which are characterized up to multiple by certain vanishing properties.

\begin{main} \label{thm:A}
    The Harish-Chandra image of $\sD_\mu$ is $k_\mu J_\mu$ where $k_\mu$ is explicitly given in Eq.~(\ref{eqn:kmu}). 
\end{main}

We will deduce this from two other results. First, let $\Zk$ be the center of $\Uk$, then we have a natural map $\pi: \Zk \hookrightarrow \Uk^\kk\twoheadrightarrow \D$, and we prove the following result.

\begin{main} \label{thm:B}
The map $\pi$ is surjective. In particular, there exist $Z_\mu \in \Zk$ such that $\pi(Z_\mu)=\sD_\mu$.
\end{main}

For $\lambda \in \cH$ let $I_\lambda$ be the generalized Verma module for $\gk$ obtained by parabolic induction from  $\qk=\kk \oplus\pplus$ of the $\kk$-summand $\Wlambda$ of $\Sk(\pplus)$ (Section~\ref{sec:GVM}). 

\begin{main}\label{thm:C}
  The central element $Z_\mu$ acts on $I_\lambda$ by $0$ unless  $\lambda_i\ge\mu_i$ for all $i$. 
\end{main}

We now sketch the key ideas in our paper.

The proof of Theorem~\ref{thm:C} is a generalization of an analogous argument in \cite{SZ2019PoSo} and relies on two key properties of $I_\lambda$: (1) it is $\kk$-spherical,  and (2) it satisfies certain branching rule (Proposition~\ref{prop:branchI(l)}). 
The module $I_\lambda$ also has a natural grading by the center of $\kk$ which leads to a slightly weaker version of Theorem~\ref{thm:C} that equally proves Theorem~\ref{thm:A}. 

For Theorem~\ref{thm:B}, we let $\hk=\ak\oplus \tk_+$ be a Cartan subalgebra of $\gk$ containing $\ak$, and let $\hcZIso:\Zk \to \Pk(\hk^*)$ be the usual Harish-Chandra homomorphism  \cite{musson2012lie}. Consider the diagram
\begin{equation} \label{diag:tight}
    \begin{tikzcd}
    \Zk \arrow[r, "\pi"] \arrow[d, "\hcZIso"'] & \D \arrow[d, "\hcIsoGK"] \\
    \Pk(\hk^*) \arrow[r, "\Res"]                            & \Pk(\ak^*)
    \end{tikzcd}
\end{equation}
where $\Res$ is the restriction map with respect to the decomposition $\hk=\ak\oplus\tk_+$. We first show, with the help of Proposition~\ref{prop:sphZariski}, that the diagram commutes, and then we prove that the image of $\hcZIso$ surjects onto the image of $\hcIsoGK$ using an explicit description of the generators of $\SEvR_{p, q}$ from \cite{Stembridge85supsym}. This implies Theorem~\ref{thm:B}. 

For Theorem~\ref{thm:A} the main point is to show that $\hcIsoGK(\sD_\mu)$ satisfies the vanishing properties that characterize $J_\mu$.  The action of $Z_\mu$ on $I_\lambda$  can be expressed in terms of its Harish-Chandra image, and thus Theorem~\ref{thm:C} implies certain vanishing properties for $\hcZIso(Z_\mu)$. Now we use Diagram (\ref{diag:tight}) to deduce the vanishing properties for $\hcIsoGK(\sD_\mu)$. 

The results of this paper represent a substantial advance over the results obtained in \cite{Z22}. 

Specifically, \cite{Z22} proves the special case $p=q=1$ of Theorem~\ref{thm:A}, and reduces the general case to $\kk$-sphericity of certain finite-dimensional $\gk$-modules \cite[Conjecture~1, Theorem~A]{Z22}. For ordinary Lie algebras, the Cartan--Helgason Theorem \cite{Helgason2} provides necessary and sufficient conditions for such $\kk$-sphericity. However, its precise analog for Lie superalgebras remains open and seems quite hard \fixIt{(see, however, \cite{ShermanCH} for recent results in this direction). }

In this paper, we offer a completely new approach to proving Theorem~\ref{thm:A} without using any conjecture. 
We show that the set of spherical weights, defined below, is Zariski dense in $\ak^*$ (Proposition~\ref{prop:sphZariski}), and this is sufficient for our purposes. Our argument is self-contained, algebraic, and independent of the results on sphericity in the literature\footnote{ see e.g. \textit{Spherical representations of Lie supergroups} by Alexander Alldridge and Sebastian Schmittner, J. Funct. Anal. 268 (2015), no. 6, 1403--1453}. 
\fixIt{One key difficulty in \cite{Z22} is that a super analog of the Sahi--Zhang branching lemma for $V_\lambda$ (\cite[Lemma~5.1]{SZ2019PoSo}), is not available. This was circumvented in \cite{Z22} by working with the cospherical vector in the restricted dual of an induced module similar to $I_\lambda$; this leads to weaker vanishing properties of $\hcIsoGK(\sD_\mu)$, which nevertheless suffice to characterize it. In the present paper, the added flexibility of using a central element ($Z_\mu$) allows us to apply the (easier) branching result for $I_\lambda$, proved in Proposition~\ref{prop:branchI(l)} below.} 

The Shimura problem discussed in this paper is closely related to the Capelli problem studied in \cite{KS1993, S94} for Lie algebras, in \cite{SaSaGLMN, AlSaSa}, and \cite{SSS2020} for Lie superalgebras, and more recently in \cite{LSS1, LSS2, LLS3} for quantum groups.  
Thus, it is possible that the results of this paper can be generalized to many of the Hermitian symmetric superpairs constructed using Jordan superalgebras, e.g., in \cite[Theorem~1.4]{SSS2020}. 
In particular, we expect our Theorem~\ref{thm:A} to hold for the pairs 
\[
(\mathfrak{osp}(4n|2m), \mathfrak{gl}(m|2n)), (\mathfrak{osp}(m+3|2n), \mathfrak{gosp}(m+1|2n)), (D(2|1,t), \mathfrak{gl}(1|2)), (F(3|1), \mathfrak{gosp}(2|4)).
\] 
These are precisely the pairs denoted $(\gk^\flat,\gk)$ in Cases II--V in \cite[Table~4]{SSS2020}. 
In addition, we hope to address the quantum version of this problem in future work.

\fixIt{It is natural to ask if the results here can be extended to more general pairs of the form $(\gl(p+r|q+s), \gl(p|q) \oplus \gl(r|s))$, however it appears that this would require additional techniques. Specifically, in our argument we need the fact that the irreducible quotient $V_\lambda$ of $I_\lambda$ is {\em finite dimensional}. However, this need not hold for more general pairs; see Remark~\ref{rmk:infDim} in Appendix~\ref{app:A}.}

The structure of the paper is as follows.
In Section~\ref{sec:prelim}, we review some known results related to the pair $(\gl(2p|2q), \gl(p|q) \oplus \gl(p|q))$. In particular, we discuss the Iwasawa decomposition, the two Harish-Chandra homomorphisms, and the Cheng--Wang decomposition of $\Sk(\pk^\pm)$.
We also present the Sergeev--Veselov supersymmetric interpolation polynomials of Type \textit{BC}.
In Section~\ref{sec:setUp}, we define the supersymmetric Shimura operators. 
In Section~\ref{sec:HCIsos}, we describe the image of $\hcIsoGK$ as the ring of even supersymmetric polynomials. We prove the surjectivity of $\Res$ in Subsection~\ref{subsec:2Im&Res} (Proposition~\ref{prop:Res}) and Theorem~\ref{thm:B} in Subsection~\ref{subsec:commDiag}.
In Section~\ref{sec:GVM} we discuss the generalized Verma modules and prove Theorem~\ref{thm:C}, which provides the necessary representation theoretic machinery for the vanishing properties.
Finally in Section~\ref{sec:BCPoly}, we reformulate the Type \textit{BC} interpolation polynomials in our context and then prove Theorem~\ref{thm:A} using the results obtained in previous sections. 
In the appendix, we prove some useful results regarding spherical and cospherical modules, in particular Proposition~\ref{prop:sphZariski} on spherical weights, which may be of independent interest.


\section{Preliminaries} \label{sec:prelim}

\subsection{Symmetric superpairs}  \label{subsec:gkHCSph}
Let $(\gk, \kk)$ be a (complex) symmetric superpair which admits an Iwasawa decomposition, where $\gk$ is basic simple or $\gk = \gl(m|n; \C)$ as in \cite{Sher22Iwasawa}.
Specifically, an involution $\theta$ on $\gk$ gives rise to the Cartan decomposition $\gk = \kk \oplus \pk$ where $\kk$ is the fixed point subalgebra and $\pk$ is the $(-1)$-eigenspace of $\theta$. 
We fix a $\theta$-invariant nondegenerate invariant form $b$ on $\gk$. 
Let $\ak \subseteq \pk\ev$ be a maximal toral subalgebra. 
Let $\Sigma := \Sigma(\gk, \ak)$ be the restricted root system of $\gk$ with respect to $\ak$, and $\Sigma^+$ be a positive system. 
We denote the form on $\ak^*$ induced from $b$ by $(\cdot, \cdot)$.
We say a restricted root $\alpha$ is anisotropic if $(\alpha, \alpha)\neq 0$, and isotropic otherwise. We write $\gk_\alpha$ for the root space of $\alpha$, and let
\[
\Sigma\ev := \{\alpha \in \Sigma : \gk _\alpha \cap \gk\ev \neq \{0\} \}, \quad \Sigma\od := \{\alpha \in \Sigma : \gk _\alpha \cap \gk\od \neq \{0\} \}
\]
be the sets of \textit{even} and \textit{odd} restricted roots respectively. Set $\Sigma^+\ev = \Sigma^+\cap \Sigma\ev$.
If $\alpha \in \Sigma$, but $\alpha/2 \notin \Sigma$, we say $\alpha$ is \textit{indivisible}.
We define the \textit{multiplicity} of $\alpha$ as
$m_\alpha := \dim (\gk_\alpha)\ev -\dim (\gk_\alpha)\od$.
For a positive system $\Sigma^+$, we define the \textit{nilpotent subalgebra} for $\Sigma^+$ as $\nk := \bigoplus_{\alpha\in\Sigma^+}\gk_\alpha$. 
We denote the \textit{restricted Weyl vector} 
$\frac{1}{2}\sum_{\alpha \in \Sigma^+} m_\alpha \alpha$
for $\Sigma^+$ as $\rho_\ak$.

An \textit{Iwasawa decomposition} for the pair $(\gk, \kk)$ is the decomposition
\begin{equation} \label{eqn:nak}
    \gk = \nk \oplus \ak \oplus \kk.
\end{equation} 
For an in-depth discussion and recent developments on Iwasawa decomposition, see \cite{Sher22Iwasawa}. 


The Poincar\'e--Birkhoff--Witt theorem applied to Eq.~(\ref{eqn:nak}) yields the following identity:
\[
\Uk = \left( \nk\Uk + \Uk\kk\right ) \oplus \Sk(\ak).
\]
Let $p$ be the corresponding projection onto $\Sk(\ak) \cong \Pk(\ak^*)$,
and define $\hcHomoGK(D)(\lambda) := p(D)(\lambda+\rho_\ak)$ on $\mathfrak{U}^\kk$. The map $\hcHomoGK$ is called the \textit{Harish-Chandra projection associated with} $(\gk, \kk)$.
The \textit{Harish-Chandra homomorphism associated with $(\gk, \kk)$} is the quotient map
\[
\hcIsoGK: \quoUk\rightarrow \Ima \hcHomoGK.
\]
In Section~\ref{sec:HCIsos}, we will show that $\ker \hcHomoGK = (\Uk\kk)^\kk$ so $\hcIsoGK$ is an isomorphism, and determine $\Ima \hcIsoGK = \Ima \hcHomoGK$. 

\subsection{The Harish-Chandra homomorphism 
for \texorpdfstring{$\Zk( \gl(m|n))$}{Z(gl(m|n))} } \label{subsec:ZHC}
Let $\gk = \gl(m|n) = \gl(m|n; \C)$ and let $\tk = \Spn\{E_{i, i}: 1\leq i \leq m+n\}$ be the standard diagonal Cartan subalgebra, where $E_{i, j}$ denotes the matrix with 1 in the $(i, j)$-th entry and 0 elsewhere. 
We denote the root system of $\gk$ with respect to $\tk$ as $\Sigma(\gk, \tk)$,
and choose a positive system $\Sigma^+(\gk, \tk)$ in $\Sigma(\gk, \tk)$. 
Let $\Nk^+$ and $\Nk^-$ be the sums of positive  and negative root spaces, respectively. 
We may then write $\gk = \Nk^- \oplus \tk \oplus \Nk^+$, the triangular decomposition of $\gk$.								   
We denote the \textit{Weyl vector} 
$\frac{1}{2}\sum_{\alpha \in \Sigma^+(\gk, \tk)} m_\alpha \alpha$ for $\Sigma^+(\gk, \tk)$, as $\rho_\tk$, where $m_\alpha$ is the multiplicity of $\alpha \in \Sigma^+(\gk, \tk)$.

We denote the universal enveloping algebra $\Uk(\gk)$ as $\Uk$, and the center of $\Uk$ as $\Zk$. Following \cite{Hum78, musson2012lie}, by the Poincar\'e--Birkhoff--Witt theorem, we have 
\begin{equation} \label{eqn:bigTriUDec}
\Uk = \Sk(\tk)\oplus (\Nk^-\Uk+\Uk\Nk^+)
\end{equation}
Let $\xi$ be the corresponding projection onto $\Sk(\tk)$. 
We identify $\Sk(\tk)$ with the polynomial algebra $\Pk(\tk^*)$, and define an automorphism $\eta$ by 
$\eta(p)(\mu+\rho_\tk)=p(\mu)$ for all $p$ in  $\Pk(\tk^*)$ and $\mu \in \tk^*$.
The \textit{Harish-Chandra homomorphism} is then defined as
\[
\hcZIso_\tk := \eta\circ\xi |_\Zk.
\]
Denote by $\chi_\mu$ the central character afforded by a $\gk$-module of highest weight $\mu \in \tk^*$, then we have
\begin{equation} 
    \chi_\mu(z)= \xi(z) = \hcZIso_\tk(z) (\mu+\rho_\tk).
\end{equation}

To describe the image of $\hcZIso_\tk$, we first introduce the supersymmetric polynomials.
Let $\{x_i\}:=\{x_i\}_{i = 1}^m$ and $\{y_j\}:=\{y_j\}_{j=1}^n$ be two sets of independent variables, and set $\{x_i, y_j\} := \{x_i\}\cup \{y_j\}$. We write $f(x_1 = t, y_1 = -t)$ for the polynomial obtained by the substitutions $x_1 = t$ and $y_1 = -t$. A polynomial $f$ in $\{x_i, y_j\}$ is said to be \textit{supersymmetric} if
\begin{enumerate}
    \item $f$ is invariant under permutations of $\{x_i\}$ and of $\{y_j\}$ separately. 
    \item $f(x_1 = t, y_1 = -t)$ is independent of $t$.
\end{enumerate}
Let $I_\C(x_i, y_j)$ denote the $\C$-algebra of supersymmetric polynomials in $\{x_i, y_j\}$.
By \cite[Theorem~1]{Stembridge85supsym} (c.f. \cite[Theorem~12.4.1]{musson2012lie}), $I_\C(x_i, y_j)$ is generated by the power sums:
\begin{equation} \label{eqn:power}
    p_r^{(m, n)}(x_{i}, y_{j}) := \sum_{i = 1}^m x_i^r -(-1)^r \sum_{j = 1}^n y_j^r,\; r\in \mathbb{Z}_{>0}.
\end{equation}
We record the following standard result about $\hcZIso_\tk$ where we identify $\{x_i, y_j\}$ as standard basis for $\tk$. Thus $x_i = E_{i, i}$ for $1\leq i \leq m$, and $y_j = E_{m+j, m+j}$ for $1\leq j \leq n$.
We refer to \cite[Theorem~13.1.1, Theorem~13.4.1]{musson2012lie}, which are based on the original works by Kac, Serganova, and Gorelik \cite{Kac84HC, Ser99HC, Gor04HC}.
\begin{thm} \label{cite:ImhcZIso}
    The homomorphism $\hcZIso_\tk$ is an isomorphism and $\Ima \hcZIso_\tk = I_\C(x_i, y_j)$.
\end{thm}

We also use the notation $\Lambda(\tk^*)$ for the algebra of supersymmetric polynomials on $\tk^*$ when we suppress the choice of coordinates.
Thus, we have
\begin{equation} \label{eqn:ZHCiso}
    \Ima \hcZIso_\tk = \Lambda(\tk^*), \quad \Zk \xrightarrow[\sim]{\hcZIso} \Lambda(\tk^*).
\end{equation}

\begin{rmk}
    The above result is true in a more general setting. In particular, Eq.~(\ref{eqn:ZHCiso}) is true when \cite[Hypothesis~8.3.4]{musson2012lie} is assumed, that is, when $\gk$ is $\gl$, or basic simple, excluding type $A(m,n)$.
\end{rmk}

\subsection{Cheng--Wang decomposition} \label{subsec:CW}
Let $\kk = \gl(m|n)\oplus \gl(m|n)$.
We recall a multiplicity-free $\kk$-module decomposition, known as Howe duality in \cite{CW2000HDfL}, which generalizes Schmid's decomposition \cite{Schmid69, FK90}. 
It will allow us to define the supersymmetric Shimura operators in Section~\ref{sec:setUp}.

We first introduce some notation. A \textit{partition} $\lambda$ is a sequence of non-negative integers $(\lambda_1, \lambda_2, \dots)$ with only finitely many non-zero terms and $\lambda_i\geq \lambda_{i+1}$ (c.f. \cite{MacdSymFun}). 
Let $|\lambda| := \sum_{i}\lambda_i$ denote the \textit{size} of $\lambda$, $\ell(\lambda) := \max\{i: \lambda_i>0\}$ the \textit{length} of $\lambda$,
and $\lambda'$ for which $\lambda'_i := |\{j: \lambda_j\geq i\}|$ the \textit{transpose} of $\lambda$. When viewed as the corresponding Young diagram, $\lambda$ is the collection of ``boxes" $(i, j)$
\[
\{(i, j): 1\leq i \leq \ell(\lambda), 1\leq j \leq \lambda_i\}.
\]
A $(m, n)$\textit{-hook partition} is a partition $\lambda$ such that $\lambda_{m+1}\leq n$. We define
\[
\cH(m, n) := \left \{\lambda: \lambda_{m+1}\leq n \right \}, \quad  \cH^d(m, n) := \left \{\lambda\in \cH(m, n):|\lambda|=d \right \}.
\]
For $\lambda \in \cH(m, n)$, we define a $(m+n)$-tuple
\begin{equation}\label{eqn:lNat}
    \lambda^\natural := \left(\lambda_1, \dots, \lambda_m,  \left \langle \lambda_1'-m \right \rangle, \dots, \left \langle \lambda_n'-m \right \rangle \right) 
\end{equation}
where $\langle x \rangle := \max \{x, 0\}$ for $x\in \mathbb{Z}$.
The last $n$ coordinates can be viewed as the lengths of the remaining columns after discarding the first $m$ rows of $\lambda$.

We let $\epu_i$ and $\dau_j$ be the coordinate of $E_{i, i}$ for $1\leq i \leq m$ and of $E_{m+j, m+j}$ for $1\leq j \leq n$ respectively. By an \textit{$\epsilon\delta$-chain}, we mean a rearrangement  $[X_1\cdots X_{m+n}]$ of $\epu_i$'s and $\dau_j$'s. The consecutive differences $X_i-X_{i+1}$ give a positive system in $\Sigma^+(\gl(m|n), \tk)$ which defines a Borel subalgebra $\mathfrak{B} \supseteq \tk$. 
For instance, $[\epu_1\cdots\epu_m\dau_1\cdots\dau_n]$
gives the standard Borel subalgebra $\bk^{\std}$, while $[\dau_n\cdots\dau_1\epu_m\cdots\epu_1]$ gives the opposite one, denoted as $\bk^{\opp}$. 
For an $(m+n)$-tuple $(a_1, \dots, a_m, b_1, \dots, b_n)$, we associate an irreducible $\gl(m|n)$-module of highest weight $\sum_{i=1}^m a_i\epu_i + \sum_{j = 1}^n b_j \dau_j$ with respect to $\bk^{\std}$, denoted as $L(a_1, \dots, a_m, b_1, \dots, b_n)$. 

The natural action of $\gl(m|n)$ on $\C^{m|n}$ gives an action of $\kk = \gl(m|n)\oplus \gl(m|n)$ on $\C^{m|n}\otimes \C^{m|n}$, which extends to an action on $\Sk(\C^{m|n}\otimes \C^{m|n})$. We record the following result regarding the $\kk$-module structure on $\Sk(\C^{m|n}\otimes \C^{m|n})$. See \cite[Theorem~3.2]{CW2000HDfL}.
\begin{thm} \label{cite:CW}
    The supersymmetric algebra $\Sk(\C^{m|n}\otimes \C^{m|n})$ as a $\kk$-module is completely reducible and multiplicity-free. In particular,
    \[
    \Sk^d(\C^{m|n}\otimes \C^{m|n}) = \bigoplus_{\lambda\in \cH^d(m, n)} L(\lambda^\natural)\otimes L(\lambda^\natural).
    \]
\end{thm}

\subsection{Supersymmetric Polynomials} \label{subsec:SVBC} 
We devote this subsection to an overview of known results of supersymmetric polynomials, including the Type \textit{BC} interpolation polynomials introduced and studied by Sergeev and Veselov in \cite{SV09}. These are super analog for the now-classic Okounkov polynomials \cite{Okounkov98, OO06}.
We specialize them to $J_\mu$ (as in Theorem~\ref{thm:A}) in Section~\ref{sec:BCPoly}.

Let $ \{ e_j  \} \cup  \{  d_i  \}$ be the standard basis for $\C^{m+n}$, and $w_j$ and $z_i$ be the coordinate functions of $e_j$ and $d_i$, for $j=1, \dots, m, i = 1\dots, n$. 
Let $\mathsf{k}$ and $\mathsf{h}$ be two parameters. Following \cite{SV09}, we assume that $\mathsf{k} \notin \mathbb{Q}_{>0}$ (called \textit{generic}), and set $\hat{\fh}$ by $2\hat{\fh}-1 = \fk^{-1}(2\fh-1)$. 
We also define
\begin{align*}
    C_\lambda^+(x; \fk) &:= \prod_{(i, j)\in \lambda} \left(\lambda_i+j+ \fk(\lambda'_j+i)+x\right), \\
    C_\lambda^-(x; \fk) &:= \prod_{(i, j)\in \lambda} \left(\lambda_i-j- \fk(\lambda'_j-i)+x\right).
\end{align*}
As in \cite[Section~6]{SV09}, we let $\rho := \sum\rho_{1j} e_j + \sum \rho_{2i} d_i$ where
\begin{equation} \label{eqn:SVrho}
    \rho_{1j} := -(\mathsf{h}+\mathsf{k}j+n), \; \rho_{2i} :=  -\mathsf{k}^{-1}\left(\mathsf{h}+\frac{1}{2}\mathsf{k}-\frac{1}{2}+i\right).
\end{equation}
This is the \textit{deformed Weyl vector} calculated with the deformed root multiplicities in \cite{SV09}.
We identify $P_{m, n} := \C[w_j, z_i]$ with $\Pk(\C^{m+n})$, and define $\Lambda^{(\fk, \fh)}_{m, n} \subseteq P_{m, n}$ as the subalgebra of polynomials $f$ which: (1) are symmetric separately in shifted variables $(w_j-\rho_{1j})$ and $(z_i-\rho_{2i})$, and invariant under their sign changes; (2) satisfy the condition 
\[
f(X-d_i+e_j) = f(X)
\]
on the hyperplane $w_j+\mathsf{k}(j-1) = \mathsf{k}z_i+i-n$.
If we equip $\C^{m+n}$ with an inner product defined by 
\[
(e_i, e_{k}) = \delta_{i, k}, \; (d_j, d_l) = \mathsf{k}\delta_{j, l}, \; (e_j, d_i) = 0,
\]
then Condition (2) is equivalent to: (2') $f(X+\alpha)=f(X)$ for any $\alpha = d_i \pm e_j$ on the hyperplane
    \[
    (X-\rho, \alpha) + \frac{1}{2}(\alpha, \alpha)=0.
    \]
For $\lambda\in \cH(m, n)$, we set $w(\lambda) = (\lambda'_1, \dots, \lambda'_n)$ and $z(\lambda) = (\left \langle \lambda_1-n \right \rangle, \dots, \left \langle \lambda_m-n \right \rangle)$. Equivalently, $(z(\lambda), w(\lambda)) = (\lambda')^\natural$ (see Eq.~(\ref{eqn:lNat})).
Proposition~6.3 in \cite{SV09} says the following.
\begin{thm}\label{cite:BCPoly}
For each $\mu \in \cH(m, n)$, there exists a unique polynomial $I^{\textup{SV}}_\mu(X; \mathsf{k, h}) \in \Lambda^{(\fk, \fh)}_{n, m}$ of degree $2|\mu|$ such that
\[
I^{\textup{SV}}_\mu(z(\lambda), w(\lambda); \mathsf{k, h}) = 0, \quad \text{for all } \lambda \nsupseteq \mu,\, \lambda \in \cH(m, n).
\]
and satisfies the normalization condition $I^{\textup{SV}}_\mu(z(\mu), w(\mu); \mathsf{k, h}) = C_\mu^-(1; \fk)C_\mu^+(2\fh - 1; \fk)$.
Moreover, $\{I^{\textup{SV}}_\mu\}$ is a basis for $\Lambda^{(\fk, \fh)}_{n, m}$.
\end{thm}

In the same work \cite{SV09}, a closely related polynomial is introduced
\begin{equation} \label{eqn:SVhat}
    I^{\textup{SV}}_\mu(z_i, w_j; \fk, \fh) := d_\mu \hat{I}_{\mu'}(z_i, w_j; \fk^{-1}, \hat{\fh}),
\end{equation}
for
\begin{equation}\label{eqn:dlam}
    d_\mu := (-1)^{|\mu|}\fk^{2|\mu|}\frac{C_\mu^-(1; \fk)}{C_\mu^-(-\fk; \fk)}.
\end{equation}
A tableau formula is provided in \cite[Proposition~6.4]{SV09} as follows.
\begin{thm} \label{cite:SVtab}
For $\lambda \in \cH(m, n)$, we have
    \begin{equation}\label{eqn:Ihat}
    \hat{I}_\lambda(z_i, w_j; \fk, \fh) = \sum_{T}\varphi_T(-\fk)\prod_{(i, j)\in \lambda} f_{T(i, j)}.
\end{equation}
\end{thm}

Here $T$ is any reverse bitableau of type $(n, m)$ and shape $\lambda \in \cH(m, n)$, with a filling by symbols $1<\cdots < n < 1'<\cdots < m'$ (see \cite[Section~6]{SV09}).
The weight $\varphi_T$ is defined as in \cite[Eq.~(41)]{SV05SuperJack}, and $f_{T(i, j)}$ has leading term $z_a^2$ if $T(i,j) = a$ for $a= 1, \dots, n$ and leading term $\fk^2 w_b^2$ if $ T(i,j) = b'$ for $b'= 1', \dots, m'$.
In fact, \cite{SV09} uses Eq.~(\ref{eqn:Ihat}) as the definition for $\hat{I}_\lambda$ and Eq.~(\ref{eqn:SVhat}) as a proposition.
Here we present the tableau formula as a theorem in parallel with the following result regarding super Jack polynomials.

We recall the theory of super Jack polynomials from \cite{SV05SuperJack}. The (monic) Jack symmetric functions  $P_\lambda(x; \theta)$ are a linear basis for the ring $\Lambda$ of symmetric functions, and the power sums $p_r$ are free generators of $\Lambda$, see e.g. \cite{MacdSymFun}.
Let $\phi_\theta$ be the homomorphism from $\Lambda$ to the polynomial ring $\C[x_i,y_j]$ which is defined on the generators as follows
\[
\phi_\theta(p_r)= p^{(m, n)}_{r, \theta} (x_i, y_j) := \sum_{i=1}^m x_i^r-\frac{1}{\theta}  \sum_{j=1}^n y_j^r.
\]
The super Jack polynomial $SP_{\lambda}$ is defined by
\[
SP_{\lambda}=SP_{\lambda}(x_i, y_j; \theta) := \phi_\theta(P_\lambda(x; \theta)).
\]

\begin{thm} \label{cite:SVThm11}
We have
\begin{equation} 
    SP_\lambda(x_i, y_j; \theta) = \sum_{T}\varphi_T(\theta)\prod_{(i, j)\in \lambda} x_{T(i, j)}
\end{equation}
where $x_{j'}$ is $y_j$ for $j=1, \dots, n$ and for $\varphi_T$ defined as in \cite[Eq.~(41)]{SV05SuperJack}.
\end{thm}


We also record the following expansion (\cite{Yan92, Stanley89}) as in \cite{SZ2019PoSo} (note $\tau$ is $\theta$ here)
\begin{equation} 
    \frac{1}{k!}\left(\sum_{i}x_i^2\right)^{k} = \sum_{|\mu|=k} C_\mu^-(1; -\theta) P_\mu(x_i^2; \theta).
\end{equation}
Applying $\phi_\theta$ on both sides, we have
\begin{equation} \label{eqn:expJack}
    \frac{1}{k!}\left(p_{2, \theta}^{(m, n)}(x_i, y_j; \theta)\right)^{k} = \sum_{\mu\in \cH^k(m, n)} C_\mu^-(1; -\theta) SP_\mu(x_i^2, y_j^2; \theta).
\end{equation}

\section{Supersymmetric Shimura Operators}
\label{sec:setUp}
From now on, we fix $\gk = \gl(2p|2q)$ and $\kk =  \gl(p|q)\oplus \gl(p|q)$. This section is devoted to the description of the pair $(\gk, \kk)$ and the important subspaces therein.
We then define the supersymmetric Shimura operators and specialize related results introduced above. 
Throughout the section we set indices $i=1, \dots, p, j = 1\dots, q$, and $k$ is either $i$ or $j$ depending on the context.

\subsection{Realization} \label{subsec:realize}
We fix the following embedding of $\kk$ into $\gk$, and identify $\kk$ with its image
\begin{equation}\label{eqn:emb}
\left( \left( \begin{array}{c|c}A_{p \times p} & B_{p \times q}\\ \hline C_{q \times p} & D_{q \times q}\end{array}\right ),\left( \begin{array}{c|c}A_{p \times p}' & B_{p \times q}'\\ \hline C_{q \times p}' & D_{q \times q}'\end{array} \right ) \right ) 
\mapsto 
\left( \begin{array}{c c |c c } A_{p \times p} & 0_{p \times p} & B_{p \times q} & 0_{p \times q} \\0_{p \times p} & A_{p \times p}' & 0_{p \times q} & B_{p\times q}' \\ \hline C_{q \times p} & 0_{q \times p} & D_{q \times q} & 0_{q \times q} \\ 0_{q \times p} & C_{q \times p}' & 0_{q \times q} & D_{q \times q}'  \end{array}\right ).
\end{equation}
We let $J := \frac{1}{2}\diag(I_{p \times p}, -I_{p \times p}, I_{q \times q}, -I_{q \times q})$, and $\theta := \Ad \exp(i\pi J)$. 
Then $\theta$ has fixed point subalgebra $\kk \cong \gl(p|q)\oplus \gl(p|q)$. 
We also have the \textit{Harish-Chandra decomposition}
\[
\gk = \pminus \oplus \kk \oplus \pplus
\]
where $\mathfrak{p}^+$ (respectively $\mathfrak{p}^-$) consists of matrices with non-zero entries only in the upper right (respectively bottom left) sub-blocks in each of the four blocks.
Set $\pk := \pminus\oplus \pplus$, so $\gk = \kk \oplus \pk$.

In our theory, we need to work with a $\theta$-stable Cartan subalgebra $\hk$ of $\gk$ cotaining a Cartan subspace $\ak \subseteq \pk\ev$. 
We present a construction of $\hk$ and $\ak$ using a certain Cayley transform as follows. 

We let $\epu_i^+:= \epu_i$, $\epu_i^-:=\epu_{p+i}$, $\dau_j^+:=\dau_j$, and $\dau_j^-:=\dau_{q+j}$
be the characters on $\tk$. 
Let
\begin{equation} \label{eqn:gammaBF}
    \gamma_i\bos := \epu_i^+-\epu_i^-, \quad \gamma_j\fer:= \dau_j^+-\dau_j^-.
\end{equation}
These are the \textit{Harish-Chandra strongly orthogonal roots}, and we denote the set of $\gamma_k\bnf$ \footnote{Here $\textsc{B}$ indicates the Boson--Boson block (top left) and $\textsc{F}$ the Fermion--Fermion block (bottom right), c.f. \cite{A2012}.} as $\HCso$.
We set $A_{i, i'} := E_{i, i'}$ for $1 \leq i, i' \leq 2p$ and $D_{j, j'} := E_{2p+j, 2p+j'}$ for $1 \leq j, j' \leq 2q$.
Associated with each $\gamma\bos_i$ is an $\sltwo$-triple spanned by $A_{i, i}-A_{p+i, p+i}, A_{i, p+i}$ and $A_{p+i, i}$ (similarly for $\gamma\fer_j$ with $D_{j, j'}$). It is not hard to see that all $(p+q)$ $\sltwo$-triples commute. 
We write $\ii$ for the imaginary unit $\sqrt{-1}$ to avoid confusion.
We also define 
\[
c\bos_i := \Ad \exp\left( {\frac{\pi}{4}}\ii (-A_{i, p+i} -A_{p+i, i} )\right), \quad c\fer_j := \Ad \exp\left( {\frac{\pi}{4}}\ii (-D_{j, q+j} -D_{q+j, j} )\right).
\]
The product
\begin{equation} \label{eqn:Cayley}
    c := \prod_i c\bos_i \prod_j c\fer_j 
\end{equation}
is thus a well-defined automorphism on $\gk$ as all terms commute. 
We set 
\begin{alignat*}{5}
    x_i &:= \ii (A_{i,p+i} - A_{p+i, i}), &&\quad x'_i &&:= A_{i,i}+A_{p+i, p+i},  &&\quad x_{\pm i} &&:= \frac{1}{2}(x'_i\pm x_i), \\
    y_j &:= \ii (D_{j, q+j} - D_{q+j,j}),
    && \quad  y'_j &&:= D_{j,j}+D_{q+j, q+j},
    && \quad  y_{\pm j} &&:= \frac{1}{2}(y'_j\pm y_j).
\end{alignat*}
By a direct (rank 1) computation, we see that under $c$:
\begin{equation} \label{eqn:rk1}
    A_{i,i} \mapsto x_{+i}, \quad A_{p+i, p+i}\mapsto x_{-i}, \quad D_{j, j} \mapsto y_{+j}, \quad D_{q+j, q+j} \mapsto y_{-j}.
\end{equation}
Recall that $\tk$ denotes the standard diagonal Cartan subalgebra. We now define
\begin{equation} \label{eqn:hCartan}
    \hk := c(\tk), \quad \ak := \Spn_\C \{x_i, y_j\}, \quad \tk_+ := \Spn_\C \{x'_i, y'_j\}. \quad (\text{Note }\hk = \ak\oplus \tk_+.)
\end{equation}
In $\tk$, the space $\tk_- := \Spn_\C \{A_{i,i}-A_{p+i, p+i}, D_{j,j} - D_{q+j, q+j}\}$ is the orthogonal complement of $\tk_+$ 
with respect to form $\frac{1}{2}\str$ on $\gk$.
Also, on $\hk$ we let
\[
\epr_i, \dar_j, \tau_i\bos, \tau_j\fer \in \hk^*
\]
be dual to $x_i, y_j, x'_i, y'_j \in \hk$ respectively. 
Note that $\tau_k\bnf$ and $\alpha_k\bnf$ vanish on $\ak$ and $\tk_+$ respectively.  
We identify $\alpha_k\bnf$ with its restriction to $\ak$.
We also have $\tau_i\bos = \frac{1}{2}(\epu_i^++\epu_i^-)$, $\tau_j\fer = \frac{1}{2}(\dau_j^++\dau_j^-)$
and we identify them with their restrictions on $\tk_+$.
On $\ak^*$, we set
\begin{equation}\label{eqn:akForm}
    (\alpha_m\bos, \alpha_n\bos)=-(\alpha_m\fer, \alpha_n\fer) = \delta_{mn}, \quad
    (\alpha_i\bos, \alpha_j\fer)=0,
\end{equation}
which is induced from $b$. 
For future purposes, we also consider the following basis for $\hk^*$:
\begin{equation} \label{eqn:chiEta}
    \chi_{\pm i} := \tau_i\bos \pm \alpha_i\bos, \quad 
    \eta_{\pm j} := \tau_j\fer \pm \alpha_j\fer.
\end{equation}

\subsection{Root Data} \label{subapp:RD}
Recall the superdimension of a super vector space $V$ is denoted as $\sdim V = \dim V\ev | \dim V\od$. Recall that for a (restricted) root 
the \textit{root multiplicity} is defined as $m_\alpha := \dim (\gk_\alpha)\ev -\dim (\gk_\alpha)\od,$
and the \textit{deformed root multiplicity} is defined as 
\[
\mult(\alpha) := -\frac{1}{2}m_\alpha.
\]

We first give an explicit description of $\Sigma(\gk, \hk)$ using the coordinates defined in Eq.~(\ref{eqn:chiEta}).
The order $<$ on the indices $\{\pm i, \pm j\}$ is interpreted as the natural order on $\mathbb{Z}$, where $+i$ means $i$.
We choose $\Sigma^+(\gk, \hk)$ so that
\begin{alignat}{2}
    \Sigma\ev^+(\gk, \hk) &= \{\chi_k - \chi_l: -p\leq k < l \leq+p\} \\
    &\cup \{\eta_k - \eta_l: -q \leq k < l \leq +q\}, \quad &&m_\alpha = 1, \label{eqn:ghEv}\\
    \Sigma\od^+(\gk, \hk) &= \{\chi_k - \eta_l:  +1 \leq k \leq +p, -q\leq l\leq +q \} & \label{eqn:ghOd}  \\
    &\cup \{\eta_l -\chi_k: -q\leq l\leq +q , -p \leq k \leq -1\}, \quad  &&m_\alpha = -1. \notag 
\end{alignat}
This can be seen by (I) viewing $\hk$ as a Weyl group conjugate of the standard Cartan subalgebra $\tk$ and identifying the roots expressed using $\epu^\pm$ and $\dau^\pm$, or (II) considering the Cayley transform which in turn gives a correspondence between positive roots on $\hk$ and those on $\tk$. Alternatively, the chain in Eq.~(\ref{eqn:hkChain}) gives the above choice of positive roots. 

We now record restricted root data taken from \cite{Z22}.
The restricted root system $\Sigma(\gk, \ak)$ is of Type \textit{C}. 
In the following table, we give the standard choice of positivity, together with the superdimensions of the restricted root spaces. 
The parameters $\mathsf{k, p, q, r, s}$ are $\mult(\alpha)$, and these are the 5 parameters first introduced by Sergeev and Veselov in \cite{SV09} to study interpolation polynomials. 
Note only the last column gives the odd restricted roots.

\begin{table}[ht] \setcounter{table}{0}
\begin{tabular}{
    |>{\centering\arraybackslash}p{45pt}
    |>{\centering\arraybackslash}p{45pt}
    |>{\centering\arraybackslash}p{65pt}
    |>{\centering\arraybackslash}p{45pt}
    |>{\centering\arraybackslash}p{45pt}
    |>{\centering\arraybackslash}p{65pt}
    |>{\centering\arraybackslash}p{65pt}|}
    \hline
     $\epr_i,\fr$ & $2\epr_i, \fs$ & $\epr_i\pm\epr_{k}, \mathsf{k}$ & $\dar_j, \fp$ & $2\dar_j, \fq$  &  $\dar_j\pm\dar_l, \mathsf{k}^{-1}$ & $\epr_i\pm\dar_j, 1$  \\
    \hline
     
     $0|0, 0$ & $1|0, -\frac{1}{2}$ & $2|0, -1$ & $0|0, 0$ & $1|0, -\frac{1}{2}$ & $2|0, -1$ & $0|2, 1$\\
    \hline
\end{tabular}
    \caption{Positive Restricted Roots and $\mult(\alpha)$. $1\leq i<k \leq p, 1\leq j<l \leq q$}  \label{tab:5para} 
\end{table}

\vspace{-25pt}
Note their $\epsilon_i$ and $\delta_p$ are our $\dar_j$ and $\epr_i$ respectively, due to a different choice of positivity (see \cite[Eq.~(71)]{SV09}).
A direct computation using Eqs.~(\ref{eqn:ghEv}) and (\ref{eqn:ghOd}) shows the Weyl vector in $\hk^*$ is
\begin{align} \label{eqn:hkRho}
    \rho_\hk &=\sum_{i=1}^p \left((p-i)+\frac{1}{2}-q\right ) (\chi_{+i}-\chi_{-i}) +\sum_{j=1}^q \left((q-i)+\frac{1}{2}\right) (\eta_{+j}-\eta_{-j}) \notag \\
    &= \sum_{i=1}^p (2(p-i)+1-2q) \epr_i +\sum_{j=1}^q (2(q-j)+1)\dar_j.
\end{align}
The Weyl vector of restricted roots in $\ak^*$ is:
\begin{equation} \label{eqn:akRho}
    \rho = \sum_{i=1}^p (2(p-i)+1-2q) \epr_i +\sum_{j=1}^q (2(q-j)+1)\dar_j,
\end{equation}
as the restriction of $\rho_\hk$ on $\ak$.

\subsection{Supersymmetric Shimura operators} \label{subsec:superSO}
Throughout the subsection, we set $i = 1, \dots, p$ and $j = 1, \dots, q$. Write $\cH$ for $\cH(p, q)$.
Recall $L(\lambda^\natural)$ is the irreducible $\gl(p|q)$-module of highest weight 
\[
\sum\lambda_i \epu_i+\sum\langle \lambda'_j-p\rangle \dau_j
\]
with respect to $\bk^{\std}$, and it is of Type \texttt{M} (\cite{CWBook}). In this case, Schur's Lemma indicates that $\dim \End_{\gl(p|q)}\left(L(\lambda^\natural)\right) = 1$, and $L(\lambda^\natural) \otimes L(\lambda^\natural)$ is actually irreducible as a $\kk$-module as $\kk = \gl(p|q)\oplus \gl(p|q)$. 
If we let $\gl(p|q)$ act on the second component contragrediently (via a negative supertranspose), then we define the irreducible $\kk$-module $L(\lambda^\natural) \otimes L^*(\lambda^\natural)$ as $\Wlambda$.
Note both $\pminus$ and $\pplus$ are $\kk$-modules by the short grading. We
identify $\pminus$ as $(\pplus)^*$ via the form $b$ also used for Eq.~(\ref{eqn:akForm}).

\begin{prop} \label{prop:goodCW}
The symmetric superalgebras $\mathfrak{S}(\pplus)$ and $\mathfrak{S}(\pminus)$ are completely reducible and multiplicity free as $\kk$-modules. Specifically,
    \begin{equation} \label{eqn:goodCW}
    \mathfrak{S}^d(\mathfrak{p}^+) = \bigoplus_{\lambda\in \cH^d} \Wlambda,\quad \mathfrak{S}^d(\mathfrak{p}^-) = \bigoplus_{\lambda\in \cH^d} \Wlambda^*.
\end{equation}
\end{prop}
\begin{proof}
    By duality, it suffices to show the first equation. 
    First we have $\pplus\cong \C^{p|q}\otimes(\C^{p|q})^*$, by identifying $\C^{p|q}$ and $(\C^{p|q})^*$ as spaces of column and row vectors respectively. 
    The contragredient $\kk$-module structure on $(\C^{p|q})^*$ is obtained by applying the negative supertranspose on $\gl(p|q)$.
    Then Theorem~\ref{cite:CW} implies
    \[
    \Sk^d(\pplus) = \bigoplus_{\lambda\in \cH^d} L(\lambda^\natural)\otimes L^*(\lambda^\natural),
    \]
    proving the claim, c.f. \cite[Theorem~1.4]{SSS2020} and the notation therein.
\end{proof}

We write down the highest weight of $\Wlambda$ with respect to the Borel subalgebra $\bk^{\std}\oplus \bk^{\opp}$ of $\kk$:
\begin{align}
      \lambda^\natural_\tk &:= \sum \lambda_i  \epu^+_i + \sum\langle \lambda_j'-p \rangle \dau^+_j -\sum \lambda_i  \epu^-_i - \sum \langle \lambda_j'-p \rangle \dau^-_j \label{eqn:naturalWt} \\
     &= \sum \lambda_i\gamma_i\bos + \sum \langle \lambda_j'-p \rangle \gamma_j\fer, \quad \gamma_i\bos, \gamma_j\fer \in \HCso.\label{eqn:gammaWt}
\end{align}
The part with negative terms is indeed dominant since we take the opposite Borel subalgebra for the second copy of $\gl(p|q)$ in $\kk = \gl(p|q)\oplus \gl(p|q)$.

As $\pk^\pm$ are supercommutative, the respective universal enveloping algebras are just $\Sk(\pk^\pm)$. The direct summand $\Wlambda^*\otimes \Wlambda$ embedded in $\mathfrak{S}(\pminus) \otimes\mathfrak{S}(\pplus)$ is then multiplied into $\Uk$. 
We write $1_\lambda$ for the element in $\left(\Wlambda^*\otimes \Wlambda\right)^\kk$ corresponding to $\mathrm{Id}_{\Wlambda}\in \End_\kk(\Wlambda)$. Let $\D := \quoUk$.

\begin{defn} \label{defn:ShimuraOp}
For each $\lambda \in \cH(p, q)$, we let $D_\lambda$ be the image corresponding to $1_\lambda$ under the composition of the multiplication and the embedding
\begin{alignat}{2} \label{eqn:longArrow}
        \left(\Wlambda^*  \otimes \Wlambda \right)^\kk \hookrightarrow \left( \Sk(\pminus)\otimes \Sk(\pplus) \right)^\kk  \rightarrow  &\Uk^\kk && \rightarrow \D \notag \\
        1_\lambda  \xmapsto{\hphantom{ 
        \left(\Wlambda^* \otimes \Wlambda \right)^\kk \hookrightarrow \left( \Sk(\pminus)\otimes \Sk(\pplus) \right)^\kk  \rightarrow  \Uk^\kk
        } }  
        & D_\lambda &&\mapsto \sD_\lambda.
\end{alignat}
The element $\sD_\lambda$ is called the \textit{supersymmetric Shimura operator associated with the partition }$\lambda$.
\end{defn}
\begin{rmk} \label{rmk:SSO}
For an irreducible Hermitian symmetric space $G/K$, such $D_\lambda$ can be similarly defined (c.f. \cite{SZ2019PoSo}) in $\D(G)$, the algebra of differential operators on $G$, identified with the universal enveloping algebra of the Lie algebra of $G$.
As $D_\lambda$ commutes with $K$, its right action descends to $\sD_\lambda \in \D$, the algebra of differential operators on $G/K$. These $\sD_\lambda$ are the original Shimura operators.
As we will study the action of $D_\lambda$ on $\gk$-modules, we call $D_\lambda\in \Uk^\kk$ as (supersymmetric) Shimura operators as well, by a slight abuse of name. 
Working with the lift $D_\lambda \in \Uk^\kk$ of $\sD_\lambda \in \Uk^\kk$ gives tremendous flexibility.
By the definitions of $\hcHomoGK$, $\hcIsoGK$, $D_\lambda$ and $\sD_\lambda$, we see
\begin{equation} \label{eqn:twoDs}
    \hcIsoGK(\sD_\lambda) = \hcHomoGK(D_\lambda).
\end{equation}
\end{rmk}

\subsection{Specialized Results} \label{subsec:special}
We now specialize some of the results we introduced in Section~\ref{sec:prelim}. 
First, we set the choice of positive roots $\Sigma^+(\gk, \tk)$ according to the chain (c.f. Subsection~\ref{subsec:CW}) 
\begin{equation}\label{eqn:chain}
    [\epu_1^+ \cdots \epu_p^+ \dau_1^+ \cdots \dau_q^+ \dau_q^- \cdots \dau_1^- \epu_p^-  \cdots \epu_1^-]
\end{equation}
which equivalently gives a Borel subalgebra of $\gk$, denoted as $\bk^\natural$, and we have
\begin{equation} \label{eqn:bNatural}
    \bk^\natural = \bk^{\std}\oplus \bk^{\opp}\oplus \pplus.
\end{equation}
Note the Cayley transform $c$ as in Eq.~(\ref{eqn:Cayley}) allows us to send the positivity on one Cartan subalgebra to that of the other Cartan subalgebra. Specifically, we have $c^{-1} : \hk \rightarrow \tk$, and the dual map $c^{-1}_*: \tk^* \rightarrow \hk^*$. 
It is then a direct computation to check that 
\begin{equation}\label{eqn:Cayley*}
    c^{-1}_*: \epu^\pm_i \mapsto \chi_{\pm i}, \quad \dau^\pm_j \mapsto \eta_{\pm j}.
\end{equation}
The notion of $\epu\dau$-chain carries over via $c$. Thus, the induced positivity of $\Sigma(\gk, \hk)$ is given by
\begin{equation} \label{eqn:hkChain}
    [\chi_{+1} \cdots \chi_{+p} \eta_{+1} \cdots \eta_{+q}\eta_{-q} \cdots \eta_{-1} \chi_{-p} \cdots \chi_{-1}].
\end{equation}
We denote the corresponding Borel subalgebra as $\bk$.

For the Harish-Chandra isomorphism $\hcZIso := \hcZIso_\hk$ on the non-diagonal $\hk$, we obtain it via the Cayley transform $c$. 
By Eq.~(\ref{eqn:rk1}), $c$ sends the standard basis for $\tk$ the basis $\{x_{\pm i}, y_{\pm j}\}$ for $\hk$. 
Thus by a change of variable using $c$, Theorem~\ref{cite:ImhcZIso} and Eq.~(\ref{eqn:ZHCiso}) imply
\begin{equation}
    \hcZIso: \Zk \rightarrow \Pk(\hk^*) = I_\C(x_{\pm i}, y_{\pm j}) = \Lambda(\hk^*).
\end{equation}
The central character $\chi_\lambda$ for $\lambda \in \hk^*$ is now given by
\begin{equation}\label{eqn:ZCharHC}
    \chi_\lambda(z) = \gamma(z)(\lambda + \rho_\hk).
\end{equation}
Specifically, $c(\Nk^+)$ corresponds to the chain in Eq.~(\ref{eqn:hkChain}) above. 

For the Harish-Chandra projection $\hcHomoGK$ and homomorphism $\hcIsoGK$, we choose $\nk$ to be compatible with $c(\Nk^-)$ (see Eq.~(\ref{eqn:bigTriUDec})) from which we have $\rho_\ak := -\rho$ and
\begin{equation} \label{eqn:g0+rho=p}
    \hcHomoGK(D)(\lambda+\rho) = p(D)(\lambda).
\end{equation}
See Eqs.~(\ref{eqn:hkRho}) and (\ref{eqn:akRho}) for the Weyl vectors $\rho_\hk$ and $\rho$. 

Let $\Bk$ be a Borel subalgebra of $\gk$ containing a Cartan subalgebra $\Hk$.
We let $V(\mu, \Bk )$ denote the irreducible $\gk$-module of $\Bk$-highest weight $\mu \in \Hk^*$. Define $V_\lambda := V(\lambda^\natural_\tk, \bk^\natural)$ for $\lambda\in \cH$ (see Eqs.~(\ref{eqn:naturalWt}), (\ref{eqn:gammaWt}) and (\ref{eqn:bNatural}) for $\lambda_\tk^\natural$ and $\bk^\natural$). A direct check of dominance condition shows that $V_\lambda$ is finite dimensional (\cite[Theorem~4.7]{Z22}). 
Proposition~3.4 in \cite{Z22}, which is essentially a rank-1 calculation of Cayley transforms, tells us how the highest weights change with respect to different Borel subalgebras. Explicitly, we have
\begin{equation} \label{eqn:naturalV}
    V_\lambda = V(\lambda^\natural_\tk, \bk^\natural) = V(2\lambda_\hk^\natural, \bk),
\end{equation}
(see Eq.~(\ref{eqn:hkChain}) for $\bk$) where
\begin{equation} \label{eqn:akEvenWt}
    \lambda_\hk^\natural := \sum_{i = 1}^p \lambda_i \alpha_i\bos+\sum_{j = 1}^q \left \langle \lambda_j'-p \right \rangle  \alpha_j\fer \in \hk^*.
\end{equation}

A $\gk$-module $V$ is said to be $\kk$-\textit{spherical} if $V^\kk := \{v\in V: X.v =0 \textup{ for all }X\in \kk \}$ is non-zero. A non-zero vector in $V^\kk$ is called a $\kk$-\textit{spherical} vector. When the context is clear, we simply say \textit{spherical} instead of $\kk$-spherical. 
It follows from definition that $V^\kk$ is $\Uk^\kk$-invariant. In particular, we record the following Lemma (see \cite[Lemma~5.1]{Z22}).
\begin{lem} \label{lem:ukScalar}
    Let $V$ be a $\gk$-module. If $\dim V^\kk = 1$, then $D\in \Uk^\kk$ acts on $V^\kk$ by a scalar.
\end{lem}
\commentout{
\begin{proof}
    Let $X\in \kk$ and $u\in \Uk^\kk$ be homogeneous. Then we have $X.(u.v) = (-1)^{|X||u|} u.(X.v)$ for any $v\in V$. If $v$ is spherical, then clearly $u.v$ is again spherical and therefore a scalar multiple of $v$. We write $u.v = c(u)v$. Then $(uu').v = u.(u'.v)$ implies that $c$ is a character.
\end{proof}
}

A result of the second author \cite[Proposition~4.1]{Z22} says that if $V$ is a finite dimensional irreducible $\kk$-spherical $\gk$-module $V$, then indeed $\dim V^\kk = 1$. 

\begin{defn} \label{def:sphWt}
    Let $\Bk$ be a Borel subalgebra of $\gk$ and $\Hk\subseteq \Bk$ be a Cartan subalgebra. We call $\mu \in \Hk^*$ a \emph{$(\Bk, \kk)$-spherical weight} if $V(\mu, \Bk)$ is $\kk$-spherical and finite dimensional.
\end{defn}

When the context is clear, we suppress the prefix $(\Bk, \kk)$. We will mostly focus on the case where $\Bk = \bk$ and $\Hk = \hk$.  
By \cite[Lemma~4.4]{Z22}, if $\lambda$ is a spherical weight, then $\lambda|_{\tk_+}=0$ and we identify $\lambda$ with an element of $\ak^*$.
\begin{prop} \label{cite:scalarClaim}
    Let $\lambda\in \ak^*$ be $(\bk, \kk)$-spherical and $D\in \mathfrak{U}^\kk$. Then $D$ acts on $V(\lambda, \bk)^\kk$ by the scalar $\hcHomoGK(D)\left(\lambda+\rho\right)$. 
\end{prop}
We postpone the proof to Appendix~\ref{app:A}; the argument is similar to \cite[Theorem~5.4]{Z22}. 

\section{Harish-Chandra Isomorphisms and a Commutative Diagram} \label{sec:HCIsos}
Let $\hcHomoGK$, $\hcIsoGK$, and  $\hcZIso$ be as in Section~\ref{sec:prelim} and Section~\ref{sec:setUp}. In this section, we prove Theorem~\ref{thm:B}. The strategy is to use a commutative diagram (\ref{diag:tight}) and the surjectivity of a certain restriction map $\Res$ defined below.
We also prove that $\ker \hcHomoGK = (\Uk\kk)^\kk$ and $\Ima \hcHomoGK = \Ima \hcIsoGK$ consists precisely of even supersymmetric polynomials on $\ak^*$, independent of \cite{Z22}.

\subsection{Image of \texorpdfstring{$\hcIsoGK$}{gamma0}}\label{subsec:2Im&Res}
Recall that the image of $ \hcZIso$ is the ring $\Lambda(\hk^*)$ of supersymmetric polynomials on $\hk^*$. 
The image of $\hcHomoGK$ (equivalently of $\hcIsoGK$), however, is less well-understood. 
Alldridge's work \cite{A2012} formulates $\Ima \hcIsoGK$ as the intersection of some subalgebras of $\Sk(\ak)$, while the second author proved a Weyl groupoid invariance formulation in \cite{Z22}. We give an independent description below.

We first introduce even supersymmetric polynomials on par with what we did in Subsection~\ref{subsec:ZHC}.
Let $\{x_i\}:=\{x_i\}_{i = 1}^m$ and $\{y_j\}:=\{y_j\}_{j=1}^n$ be two sets of independent variables, and set $\{x_i, y_j\} := \{x_i\}\cup \{y_j\}$. We write $f(x_1 = t, y_1 = -t)$ for the polynomial obtained by the substitutions $x_1 = t$ and $y_1 = -t$. A polynomial $f$ in $\{x_i, y_j\}$ is said to be \textit{even supersymmetric} if (1) $f$ is invariant under permutations of $\{x_i\}$ and of $\{y_j\}$ separately, and invariant under sign changes of $\{x_i, y_j\}$ and (2) $f(x_1 = t, y_1 = -t)$ is independent of $t$.
Denote the algebra of even supersymmetric polynomials in $\{x_i, y_j\}$ as $\SEvR_{m, n}$.
It is not hard to see that an even supersymmetric polynomial is supersymmetric in $\{x_i^2, y_j^2\}$, so $\SEvR_{m, n}$ is generated by $p_{2r}(m, n)$ in Eq.~(\ref{eqn:power}) for $r\in \mathbb{Z}_{>0}$.

Let $\{\epr_i\}_{i=1}^p\cup\{\dar_j\}_{j=1}^q$ be the standard basis for $\ak^*$, equipped with the inner product $(\cdot, \cdot)$ as in Eq.~(\ref{eqn:akForm}).
Thus $\{x_i\}$ and $\{y_j\}$ as in Subsection~\ref{subsec:realize} are the coordinates of this basis.
The Condition (2) above is equivalent to (see \cite{musson2012lie}) (2') $f(X+\epr_1-\dar_1) = f(X)$ if $(X, \epr_1-\dar_1) = 0$.
As such, $\SEvR_{m, n}$ is identified as a subalgebra of $\Pk(\ak^*)$, and denoted as $\RingEv(\ak^*)$ (c.f. $\Lambda(\hk^*)$).

Let $W_0$ be the Weyl group associated with the restricted root system $\Sigma := \Sigma(\gk, \ak)$ (see Subsection~\ref{subapp:RD}). 
Then $W_0$ is of Type \textit{BC} and equals $(\mathscr{S}_p \ltimes (\mathbb{Z}/2\mathbb{Z})^p)\times (\mathscr{S}_q \ltimes (\mathbb{Z}/2\mathbb{Z})^q)$.
Originally in \cite{A2012}, $\Ima \hcHomoGK$ is given by the intersection of certain subspaces using explicitly chosen elements.
In \cite{Z22}, it is proved that this description is independent of any choice, and we have
\begin{equation} \label{eqn:UkHCiso}
     \Ima \hcHomoGK = \RingEv(\ak^*), \quad
     \quoUk \xrightarrow[\sim]{\hcIsoGK} \RingEv(\ak^*).
\end{equation}
We will offer a different proof that $\ker \hcHomoGK = (\Uk\kk)^\kk$ and $\Ima \hcHomoGK = \Ima \hcIsoGK = \RingEv(\ak^*)$ below in Subsection~\ref{subsec:commDiag}.

\begin{rmk}
    See \cite{SV2011Grob, Z22} for an alternative way of writing down (even) supersymmetric polynomials. Namely, we can consider a so-called Weyl groupoid that captures both the usual double symmetry on two sets of coordinates and the translational invariance on hyperplanes for all isotropic (restricted) roots (as in Condition (2')). 
    If $\mathfrak{W}$ (respectively $\mathfrak{W}_0$) denotes the Weyl groupoid associated with $\Sigma(\gk, \hk)$ (respectively $\Sigma(\gk, \ak)$), then $\Ima \hcZIso = \mathfrak{P}(\hk^*)^{\mathfrak{W}}$ while $\Ima \hcIsoGK =  \Pk(\ak^*)^{\mathfrak{W}_{0}}$.
\end{rmk}


Recall $\hk = \tk_+\oplus\ak$.
The respective projection onto $\ak$ can be extended to a projection from $\Sk(\hk)$ to $\Sk(\ak)$, denoted as $\Res$.
With the identification $\Sk(V)\cong \Pk(V^*)$, $\Res$ is exactly the restriction of a polynomial $f$ defined on $\hk^*$ to a polynomial $g$ defined on $\ak^*$, so $\Res(f) := g$.
Next, we show that $\Res$ on $\Lambda(\hk^*)$ surjects onto $\RingEv(\ak^*)$. Loosely speaking, the non-restricted supersymmetry should cover the restricted supersymmetry entirely.

\begin{prop} \label{prop:Res}
    The restriction map $\Res$ on $\Lambda(\hk^*)$ surjects onto $\RingEv(\ak^*)$.
\end{prop}

\begin{proof}
    We show that $\Res$ maps generators of $\Lambda(\hk^*)$ to all the generators of $\RingEv(\ak^*)$.
    For this, we choose coordinate systems for both $\Pk(\hk^*)$ and $\Pk(\ak^*)$ as in Section~\ref{sec:setUp}.
    Throughout the proof we set indices $i=1, \dots, p, j = 1\dots, q$.
    Recall $x_i$, $y_j$, $x_i'$, $y_j'$, $x_{\pm i}$, and $y_{\pm j}$ from Subsection~\ref{subsec:realize}. Also, $x_{\pm i}$ and $y_{\pm j}$ span $\hk$ while $x_i$ and $y_j$ span $\ak$ by Eq.~(\ref{eqn:hCartan}). 
    Thus by definition, the homomorphism $\Res$ on $\Pk(\hk^*) = \C[x_{\pm i}, y_{\pm j}]$ makes the substitution
    \begin{equation} \label{eqn:halfProj}
        x_{\pm i} \mapsto \pm\frac{1}{2}x_i, \quad y_{\pm j} \mapsto \pm\frac{1}{2}y_j.
    \end{equation}
    The generators in Eq.~(\ref{eqn:power}) of  $\Lambda(\hk^*)$ become $p_r^{(2p, 2q)}(x_{\pm i}, y_{\pm j})$.
    By Eq.~(\ref{eqn:halfProj}), we have
    \begin{equation*}
        \Res(p_r^{(2p, 2q)}(x_{\pm i}, y_{\pm j}) ) = 
        \begin{cases}
        2^{1-r} p_{2r}^{(p, q)}(x_i, y_j), & \text{ if } r \text{ is even }\\ 
        0, & \text{ if } r \text{ is odd}
        \end{cases}.
    \end{equation*}
    The preimage of $p_{2r}^{(p, q)}(x_i, y_j)$ under $\Res$ is $2^{r-1}p_{2r}^{(2p, 2q)}(x_{\pm i}, y_{\pm j})$,
    meaning the generators of $\RingEv$ are in the image of $\Res$.
    Therefore, $\Ima \Res = \RingEv(\ak^*)$.
    \end{proof}

    The substitution in Eq.~(\ref{eqn:halfProj}) also tells us that $f\in \Lambda(\hk^*)$ and $\Res(f) \in \RingEv(\ak^*)$ are related by
    \begin{equation} \label{eqn:resEval}
        f\left(\sum_i a_i(\chi_{+i}-\chi_{-i})+\sum_j b_j(\eta_{+j}-\eta_{-j})\right) = (\Res(f))\left(\sum_i 2a_i\alpha_i\bos +\sum_j 2b_j\alpha_j\fer\right) .
    \end{equation}

\subsection{A commutative diagram} \label{subsec:commDiag}	  
The center $\Zk$ acts on $V(\lambda, \bk)$ by a character via $\hcZIso$. Considered as a subalgebra of $\Uk^\kk$, there is another way to compute such a character using $\hcIsoGK$ if $V(\lambda, \bk)$ is spherical.
We aim to relate the two Harish-Chandra isomorphisms $\hcZIso$ and $\hcIsoGK$. 

We denote the composition map $\Zk \hookrightarrow \Uk^\kk \twoheadrightarrow \quoUk$ as $\pi$.
Note by definition, $(\Uk\kk)^\kk\subseteq \Uk\kk$ is contained in the kernel of the Harish-Chandra projection, hence in $\ker \hcHomoGK$. So the homomorphism $\hcIsoGK: \D = \quoUk \rightarrow\Pk(\ak^*)$ is well-defined and $\Ima \hcHomoGK = \Ima \hcIsoGK$.
We form the following diagram first. 
\begin{equation} \label{diag:loose}
    \begin{tikzcd}
        \Zk \arrow[r, "\pi"] \arrow[d, "\hcZIso"', "\wr"] & \D \arrow[d, "\hcIsoGK "] \\
        \Lambda(\hk^*) \arrow[r, "\Res"]                            & \Pk(\ak^*)
    \end{tikzcd}
\end{equation}
Now Theorem~\ref{thm:B} follows from the two assertions below. 

\begin{prop} \label{prop:commDiag}
    The diagram \textup{(\ref{diag:loose})} is commutative.
\end{prop}

\begin{prop} \label{prop:gamma'}
    The Harish-Chandra homomorphism $\hcIsoGK$ is an isomorphism  and $\Ima\hcIsoGK = \RingEv(\ak^*)$.
\end{prop}

\begin{proof} [Proof of Theorem~\ref{thm:B}]
    By Propositions~\ref{prop:commDiag} and \ref{prop:gamma'}, the above diagram (\ref{diag:loose}) gives the diagram (\ref{diag:tight}):
    \begin{equation*}
    \begin{tikzcd}
    \Zk \arrow[r, "\pi"] \arrow[d, "\hcZIso"', "\wr"] & \D \arrow[d, "\hcIsoGK", "\wr"'] \\
    \Lambda(\hk^*) \arrow[r, "\Res"]                            & \RingEv(\ak^*)
    \end{tikzcd}
\end{equation*}
By Proposition~\ref{prop:Res}, $\Res$ is surjective. Thus the homomorphism $\pi$ is also surjective onto $\D$.
\end{proof}

Now we give the proofs of Propositions~\ref{prop:commDiag} and \ref{prop:gamma'}.

\begin{proof} [Proof of Proposition~\ref{prop:commDiag}] 
   We prove that $\Res(\hcZIso(z)) = \hcIsoGK(\pi(z))$ for any $z\in \Zk$. Let $S$ be the set of $(\bk, \kk)$-spherical weights and let $V:=V(\lambda, \bk)$ for some $\lambda\in S$. The set $S$ is Zariski dense in $\ak^*$ by Proposition~\ref{prop:sphZariski}. 
   By Proposition~\ref{cite:scalarClaim}, $z\in \Zk \subseteq \Uk^\kk$ acts on $V^\kk$ by $\hcHomoGK(z)(\lambda+\rho)$ which is $\hcIsoGK(\pi(z))(\lambda+\rho)$.
   On the other hand, by Eq.~(\ref{eqn:ZCharHC}), $z$ acts on the entirety of $V$ by $\hcZIso(z)(\lambda+\rho_\hk)$. Hence
   \begin{equation} \label{eqn:charSame}
   \hcZIso(z)(\lambda+\rho_\hk) = \hcIsoGK(\pi (z))(\lambda+\rho).
   \end{equation}
   Since $\lambda+\rho_\hk$ vanishes on $\tk_+$ (see Eqs.~(\ref{eqn:hkRho}) and (\ref{eqn:akRho})), the left side becomes $\Res(\hcZIso(z))(\lambda+\rho)$.
   Thus 
   \[
   \Res(\hcZIso(z))(\lambda+\rho) = \hcIsoGK(\pi(z))(\lambda+\rho), 
   \]
   for all $\lambda$ in $S$. Since $S$ is Zariski dense we get $\Res(\hcZIso(z)) = \hcIsoGK(\pi(z))$. 
\end{proof}

\begin{proof}[Proof of Proposition~\ref{prop:gamma'}]
    Now we show that $\Ima \hcIsoGK = \Ima \hcHomoGK = \RingEv(\ak^*)$, and $\hcIsoGK$ is injective. 
   Recall that $\cH = \cH(p, q)$. We let $\cH_d := \{\lambda\in \cH : |\lambda|\leq d\}$.
   We also let $\Pk_{2d}(\ak^*) := \{f\in \Pk(\ak^*): \deg f \leq 2d\}$, and $\RingEv_d(\ak^*) := \{f\in \RingEv(\ak^*):\deg f \leq 2d\}$.
   
   \textit{Step 1.} 
   Let $\Pk := \Sk(\pminus)\Sk(\pplus)$.
   By the Poincar\'e--Birkhoff--Witt Theorem, we have $\Uk = \Uk\kk \oplus \Pk$. This is in fact a $\kk$-module decomposition, and for the $\kk$-invariants, we have 
   \[
   \Uk^\kk = (\Uk\kk)^\kk \oplus \Pk^\kk.
   \]
   By Definition~\ref{defn:ShimuraOp}, we see that $D_\lambda \in (\Wlambda^*\Wlambda)^\kk$ constitute a basis for 
   \begin{equation}\label{eqn:SPKdecomp}
   \Pk^\kk = \bigoplus_{\lambda \in \cH} (\Wlambda^*\Wlambda)^\kk
   \end{equation}
   indexed by $\lambda \in \cH$. Let $\sD_\lambda \in \quoUk = \D$ be the equivalent class of $D_\lambda\in \Uk^\kk$.
   Identifying $\Pk^\kk \cong \quoUk$, we see that $\sD_\lambda$ gives a basis for $\D$.
   
   \textit{Step 2.}
   Furthermore, we see that $\Pk^\kk = \bigoplus_{d\geq 0}(\Sk^{d}(\pminus)\Sk^{d}(\pplus))^\kk$ is graded by non-negative integers according to $|\lambda|$ in Eq.~(\ref{eqn:SPKdecomp}) (c.f. Eq.~(\ref{eqn:goodCW})) as a vector superspace.
   The basis $\{\sD_\lambda\}$ is homogeneous. We further define
   \[
   \D_d  := \Spn \{\sD_\lambda: \lambda\in \cH_d\}.
   \]
   This gives a vector superspace filtration. The reason why we consider filtration instead of grading is that $\hcIsoGK$ only preserves filtration due to the $\rho$ shift in its definition. 
   For each filtered degree, let
   \[
   \hcIsoGK_d : \D_d \rightarrow \Pk_{2d}(\ak^*).
   \]
   
   \textit{Step 3.}
   By Proposition~\ref{prop:Res}, $\Res$ is surjective onto $\RingEv(\ak^*)$, from which we have
   \[
   \RingEv(\ak^*) = \Res(\hcZIso(\Zk))
   \]
   By Proposition~\ref{prop:commDiag}, the right side is $\hcIsoGK(\pi(\Zk))$, which gives $\RingEv(\ak^*)\subseteq \Ima \hcIsoGK$. By \textit{Step 2.}, the filtered version of the assertion is
   \begin{equation} \label{eqn:filteredLambdaincluded}
   \RingEv_d(\ak^*) \subseteq \Ima \hcIsoGK_d.
   \end{equation}
   
   \textit{Step 4.}
   By Proposition~\ref{prop:noShift} (independent of this result) below, the polynomials $J_\lambda$ with $|\lambda|\leq d$ give a basis for $\RingEv_d(\ak^*)$. 
   Such $J_\lambda$ are indexed again by $\lambda \in \cH_d$, so $ \dim \RingEv_d(\ak^*) = |\cH_d|$. Also, since $\dim \D_d = |\cH_d|$, we have $\dim \Ima \hcIsoGK_d \leq |\cH_d|$.
   Therefore, by Eq.~(\ref{eqn:filteredLambdaincluded}) we obtain
   \[
   \RingEv_d(\ak^*) = \Ima \hcIsoGK_d.
   \]
   This shows $\Ima \hcIsoGK = \Ima \hcHomoGK = \RingEv(\ak^*)$, as well as the injectivity of $\hcIsoGK$.
   Therefore, $\hcIsoGK$ is bijective, from which we have $\hcIsoGK:\D \overset{\sim}{\rightarrow} \RingEv(\ak^*)$ and that $\ker \hcHomoGK = (\Uk\kk)^\kk$.
\end{proof}

\section{Generalized Verma Modules} \label{sec:GVM}
In this section, we study certain generalized Verma modules, which we denote as $I_\lambda$. 
These are $\gk$-modules obtained by parabolic induction on $\Wlambda$ ($\lambda\in \cH(p, q)$) with a parabolic subalgebra containing $\kk$. 
We show that $I_\lambda$ is spherical.
We investigate a natural grading on $I_\lambda$ and use this to show that certain central elements $Z_\mu$ act trivially on $I_\lambda$  whenever $|\lambda| \leq |\mu|$.
A strengthened result is also obtained, asserting the same vanishing action of $Z_\mu$ for any $\lambda \nsupseteq \mu$. This is Theorem~\ref{thm:C}.

\subsection{Basics} \label{subsec:GVMprops}
Consider the parabolic subalgebra $\qk := \kk \oplus \pplus$ in $\gk$. The associated set of roots is given by $\Sigma(\kk, \tk)\cup \Sigma(\pplus, \tk)$. 
Let $\Wlambda$ be as in the Cheng--Wang decomposition. We extend the $\kk$-action trivially to $\pplus$ to obtain a $\qk$-module structure on $\Wlambda$.
We define the \textit{generalized Verma module} as 
\[
I_\lambda := \Ind_{\kk+\pplus}^\gk \Wlambda = \Uk \otimes_{\Uk(\qk)} \Wlambda.
\]
By the Poincar\'e--Birkhoff--Witt theorem, we see that $\Uk = \Sk(\pminus)\Uk(\kk)\Sk(\pplus)$, and we have
\[
I_\lambda \cong \Sk(\pminus) \otimes \Wlambda 
\]
both as super vector spaces and as $\kk$-modules. Clearly, $I_\lambda$ is a weight module and the highest weight is Eq.~(\ref{eqn:naturalWt}), given with respect to $\bk^\natural$ (see Eq.~(\ref{eqn:bNatural})),
and $V_\lambda$ is the irreducible quotient of $I_\lambda$.

Next, we introduce a grading on $I_\lambda$. Recall that in the one-dimensional center of $\kk$, there is a special element $J$ (see Section~\ref{sec:setUp}) such that the Harish-Chandra decomposition $\gk = \pminus \oplus \kk \oplus \pplus$
corresponds to $(-1, 0, 1)$-eigenspaces of $\ad J$, and gives rise to the short grading. 
We may extend the action of $\ad J$ to $\Uk = \Sk(\pminus) \Uk(\kk) \Sk(\pplus)$. 
Let $\{\xi_i\}, \{x_i\}$ and $\{\eta_i\}$ be homogeneous bases for $\pminus$, $\kk$, and $\pplus$ respectively, 
with which we write down the Poincar\'e--Birkhoff--Witt basis:
\[
\left\{\xi_{i_1}\cdots \xi_{i_m} x_{j_1}\cdots x_{j_l}\eta_{k_1}\cdots \eta_{k_n}: i_1 \leq \cdots \leq i_m, j_1\leq \cdots \leq j_l, k_1 \leq \cdots \leq k_n\right\}.
\]
where the index of an odd vector occurs at most once. 
We let 
\[
\Uk^d := \Spn\{\xi_{i_1}\cdots \xi_{i_m} x_{j_1}\cdots x_{j_l}\eta_{k_1}\cdots \eta_{k_n}: n-m = d\}.
\]

\begin{lem} \label{lem:UEAgraded}
    The universal enveloping algebra $\Uk$ is graded by $\ad J$, and $\Uk = \bigoplus_{d\in \mathbb{Z}} \Uk^d$.
\end{lem}
\begin{proof}
    The action of $J$ on the basis $\xi_{i_1}\cdots \xi_{i_m} x_{j_1}\cdots x_{j_l}\eta_{k_1}\cdots \eta_{k_n}$ is given by the scalar $n-m = d$ in $\Uk^d$. The bases for $\Uk^d$ for all $d\geq 0$ give a partition of the Poincar\'e--Birkhoff--Witt basis, and the sum is hence direct.
    Furthermore, the multiplication in $\Uk$ respects the Lie superbracket, which in turn respects the short grading on $\gk$. Therefore the grading is well-defined.
\end{proof}

Let $l = |\lambda|$. The module $I_\lambda$ is also endowed with the $\ad J$ grading. Explicitly, 
\[
I_\lambda^{l-d} := \Sk^d(\pminus) \otimes \Wlambda
\]
has degree $l-d$ for $d \geq 0$.
In particular, its top homogeneous component is $\C\otimes \Wlambda$ of degree $l$, isomorphic to $\Wlambda$.
By Lemma~\ref{lem:UEAgraded}, since the $\gk$ acts by left multiplication on $I_\lambda$, we have $\Uk^d.I_\lambda^m \subseteq I_\lambda^{d+m}$. This leads to the following lemma.

\begin{lem} \label{lem:GVMdeg}
    If $u\in \Sk^k(\pplus)$, then $u.I_\lambda^m \subseteq I_\lambda^{k+m}$.
\end{lem}

Next, we show that $I_\lambda$ is spherical.

\begin{prop} \label{prop:GVMsph}
The module $I_\lambda$ is spherical. Furthermore, wehave $\dim I_\lambda^\kk = 1$ and $I_\lambda^\kk \subseteq I_\lambda^0$. 
\end{prop}
\begin{proof}
As a $\kk$-module, we have
$I_\lambda \cong \bigoplus_{\nu\in \cH} W_\nu^* \otimes \Wlambda$ and each component can be identified as $\Hom(W_\nu, \Wlambda)$.
Only the degree 0 component $\Wlambda^* \otimes \Wlambda \cong \End \left(\Wlambda\right)$ has a one-dimensional $\kk$-invariant subspace by Schur's Lemma for Type \texttt{M} modules. 
Hence $I_\lambda$ is spherical with a unique up to constant spherical vector. 
\end{proof}

By the above Proposition, we fix a non-zero $\omega \in I_\lambda^\kk$ in the following discussion.

\subsection{Vanishing actions} \label{subsec:vanAct}
Recall that the supersymmetric Shimura operator $D_\mu$ in $\Uk^\kk$ is defined as the image of $1_\mu$ under the following composition of maps (Definition~\ref{defn:ShimuraOp})
\begin{align*}
        \left(\Wmu^* \right. \otimes \left.\Wmu \right)^\kk \hookrightarrow \left( \Sk(\pminus)\otimes \Sk(\pplus) \right)^\kk  \rightarrow  &\Uk^\kk \\
        1_\mu  \xmapsto{\hphantom{ 
        \left(\Wmu^* \right. \otimes \left.\Wmu \right)^\kk \hookrightarrow \left( \Sk(\pminus)\otimes \Sk(\pplus) \right)^\kk  \rightarrow  \Uk^\kk
        } }  
        & D_\mu
\end{align*}
In particular $D_\mu \in \Wmu^* \Wmu$.
If we take a homogeneous basis $\{\eta^\mu_i\}$ for $\Wmu$
and its dual basis $\{\xi^{-\mu}_i\}$ for $\Wmu^*$, then $D_\mu$ can be written as
\begin{equation} \label{eqn:Dmu}
    D_\mu = \sum \xi^{-\mu}_i \eta^\mu_i.
\end{equation}

Let $m = |\mu|$ and $l = |\lambda|$. We have the following ``vanishing" actions (Propositions~\ref{prop:l<m} and \ref{prop:l=m}).

\begin{prop} \label{prop:l<m}
    If $l<m$, then $D_\mu.v^0 = 0$ for any $v^0\in I_\lambda^0$.
\end{prop}
\begin{proof}
    In $D_\mu$, every $\eta^\mu_i$ in Eq.~(\ref{eqn:Dmu}) has degree $m$ as an element in $\Sk^m(\pplus)$. Hence $\eta^\mu_i.I_\lambda^0 \subseteq I_\lambda^m$ by Lemma \ref{lem:GVMdeg}. The top degree of $I_\lambda$ is $l$. But $l<m$. Thus $D_\mu.I_\lambda^0 = I_\lambda^m=\{0\}$.
\end{proof}

Viewing $\Wmu \subseteq \Uk$, we have a $\kk$-action map $R: \Wmu \otimes I_\lambda \rightarrow I_\lambda$ which is given by $\eta \otimes v \mapsto \eta.v$ for $\eta \in \Wmu$ and $v \in I_\lambda$.

\begin{prop} \label{prop:l=m}
    If $l = m$ but $\lambda \neq \mu$, then $D_\mu.\omega = 0$.
\end{prop}
\begin{proof} The action map $R$ maps $\Wmu \otimes I_\lambda^\kk\cong W_\mu$ to $I_\lambda^l = \C\otimes \Wlambda\cong \Wlambda $. Since $\lambda \neq \mu$, this map is $0$ by Schur's lemma. It follows that the iterated action map on $\Wmu^*\otimes \Wmu \otimes I_\lambda^\kk$ is $0$; hence $D_\mu.\omega =0.$
\end{proof}

\begin{cor} \label{cor:vanishingAction}
    For any $\lambda \in \cH(p, q)$ such that $|\lambda| \leq |\mu|$, $\lambda\neq \mu$, we have $D_\mu.\omega = 0$ in $I_\lambda$.
\end{cor}
\begin{proof}
    It follows from Propositions \ref{prop:GVMsph} and \ref{prop:l<m} for the case $|\lambda|<|\mu|$, and Proposition~\ref{prop:l=m} for the case $|\lambda| = |\mu|$.
\end{proof}

Recall $\pi(D_\mu) = \sD_\mu$.
By Theorem~\ref{thm:B}, $\pi$ is surjective. 
Then it is possible to find $Z_\mu\in \Zk$ such that $\pi(Z_\mu) = \sD_\mu$. 
We show that $Z_\mu$ acts by 0 on $I_\lambda$ when $\lambda$ satisfies certain conditions.

\begin{prop} \label{prop:weakC}
    The central element $Z_\mu$ acts on $I_\lambda$ by $0$ when  $|\lambda| \leq |\mu|$, and $\lambda \neq \mu$.
\end{prop}

\begin{proof}
By Lemma~\ref{lem:ukScalar}, $D_\mu$ acts on $\omega \in I_\lambda$ by a scalar $c_\mu(\lambda)$ for any $\lambda$. 
On the other hand, the action of $X\in \Uk^\kk$ on a spherical vector $v^\kk$ in a $\gk$-module descends to an action of $\D = \quoUk$ by setting $\pi(X).v^\kk := X.v^\kk$. 
Hence, in $I_\lambda$, there is no ambiguity of writing 
\begin{equation} \label{eqn:sameAction}
    c_\mu(\lambda)\omega = D_\mu.\omega = \sD_\mu.\omega = Z_\mu.\omega.
\end{equation}
Therefore, by Corollary~\ref{cor:vanishingAction},
\begin{equation} \label{eqn:cIsZero}
    c_{\mu}(\lambda) = 0, \; \text{if } |\lambda| \leq |\mu|, \lambda \neq \mu.
\end{equation}
Since $Z_\mu \in \Zk$ acts on the entirety of a highest weight module by a scalar and $I_\lambda$ is a highest weight module, such scalar has to be $c_\mu(\lambda)$ which vanishes according to Eq.~(\ref{eqn:cIsZero}).
\end{proof}

\subsection{A stronger result} What is proved in the previous subsection turns out to be enough for this paper (Theorem~\ref{thm:A}). Nonetheless, we give a better description of such vanishing actions, proving Theorem~\ref{thm:C}.

First, we recall the following ``weight decomposition" proposition, parallel to Proposition~3.2 in \cite{K10Tensor} (c.f. \cite[Theorem~5.1]{Ko59Weight}). 

\begin{prop} \label{prop:wtDecomp}
 Let $\Lk$ be a Lie (super)algebra with Borel subalgebra $\bk=\hk+\nk$ and triangular decomposition $\Lk= \nk^{-}+\bk$. Let $L_1,L_2,V$ be $\Lk$-modules such that $L_1,L_2$ are highest weight modules with $\bk$-highest weights $\kappa_1, \kappa_2\in \hk^*$, and $L_2$ is irreducible and finite dimensional. If $\Hom_{\Lk} \left(L_1, L_2\otimes V \right) \neq \{0\}$, then $\Hom_\hk \left(\C_{\kappa_1-\kappa_2}, V  \right) \neq \{0\}$, that is, $\kappa_1-\kappa_2$ is an $\hk$-weight of $V$.
\end{prop}
\begin{proof} We have the following inclusions:
\begin{align*}
    \Hom_{\Lk} \left(L_1, L_2\otimes V \right)
    &\hookrightarrow \Hom_{\bk} \left(\mathbb{C}_{\kappa_1}, L_2 \otimes V\right) \cong  \Hom_{\bk} \left(\mathbb{C}_{\kappa_1} \otimes L_2^*,  V\right)\\
    &\hookrightarrow \Hom_\hk \left(\C_{\kappa_1} \otimes \C_{-\kappa_2}, V  \right) \cong \Hom_\hk \left(\C_{\kappa_1-\kappa_2}, V  \right).
\end{align*}
The first inclusion follows since the one-dimensional $\bk$-module $\mathbb{C}_{\kappa_1}$ generates $L_1$ as an $\nk^{-}$ module, while the second follows since the $\hk$-module $\mathbb{C}_{-\kappa_2}$ generates $L_2^*$ as an $\nk$-module.
\end{proof}

We would like to show the following branching statement, i.e., $\Wmu$ occurs in $I_\lambda$ as a homomorphic image \fixIt{only} if $\lambda_i \geq \mu_i$, i.e. $\lambda\supseteq \mu$.

\begin{prop} \label{prop:branchI(l)}
    If $\Hom_\kk \left(\Wmu, I_\lambda\right) \neq \{0\}$, then $\lambda \supseteq \mu$.
\end{prop}
\begin{proof}
    The $\kk$-module structure on $I_\lambda$ is given by $\Sk(\pminus) \otimes \Wlambda$.
    Let us denote the representation map on $\Sk(\pminus) \otimes \Wlambda$ (respectively $\Wlambda \otimes \Sk(\pminus)$) as $\pi$ (respectively $\pi'$).
    The braiding map $s: \Sk(\pminus) \otimes \Wlambda \rightarrow \Wlambda \otimes \Sk(\pminus)$ is a $\kk$-module isomorphism, meaning that if $\Wmu$ occurs in $I_\lambda$, then it also occurs in $\Wlambda \otimes \Sk(\pminus)$. 
\commentout{that is, $s$ intertwines with the representation maps:
    \[
    \pi'(X) \circ s = s \circ \pi(X)
    \]
    for all $X\in \kk$,}
    
    By Proposition~\ref{prop:wtDecomp}, $ \nu= \mu^\natural_\tk - \lambda^\natural_\tk$ is a weight of $\Sk(\pminus)$. Here, $\mu^\natural_\tk$ and $\lambda^\natural_\tk$ are weights in $\tk^*$. As both $\mu^\natural_\tk$ and $\lambda^\natural_\tk$ are integral combinations of $\gamma\in \HCso$, so is $\nu$. Since $\HCso \subseteq \Sigma(\pplus\ev, \tk)$, such integral coefficients in $\nu$ must be non-positive. Equivalently, $\lambda\supseteq \mu$.
\end{proof}

\begin{proof}[Proof of Theorem~\ref{thm:C}]
    As in the proof of Proposition~\ref{prop:weakC}, $c_\mu(\lambda)$ is the scalar by which $D_\mu$ acts on $I_\lambda^\kk$. 
    Consider the $\kk$-action map $R$ as above. When restricted to $W_\mu \otimes I_\lambda^\kk \cong W_\mu$, the image $M$ is a quotient of $W_\mu$.
    If $\lambda \nsupseteq \mu$, then $M$ is zero by Proposition~\ref{prop:branchI(l)}. 			
    Spelling out $D_\mu$ as in Eq.~(\ref{eqn:Dmu}), we see that $D_\mu.I_\lambda^\kk \subseteq \Wmu.I_\lambda^\kk = \{0\}$.
    Thus,
    \begin{equation} \label{eqn:evenBetter}
        c_{\mu}(\lambda) = 0, \; \text{ for all } \lambda \nsupseteq \mu,
    \end{equation}
    and this proves that $Z_\mu$ acts on $I_\lambda$ by 0 unless $\lambda_i\ge\mu_i$ for all $i$. 
\end{proof}

Eq.~(\ref{eqn:evenBetter}) gives more vanishing points of $c_\mu$ than we actually need for proving Theorem~\ref{thm:A} in Section~\ref{sec:BCPoly}. It is stronger than Proposition~\ref{prop:weakC} which concerns only zeros at lower degrees, and is also closer to the original formulation of the vanishing properties, introduced in Theorem~\ref{cite:BCPoly}, c.f. Proposition~\ref{prop:noShift} below.

\section{Type BC Interpolation Polynomials} \label{sec:BCPoly}
In this section, we specify Sergeev and Veselov's results (Subsection~\ref{subsec:SVBC}) with all the parameters in Table~\ref{tab:5para} (Subsection~\ref{subapp:RD}), and then an explicit change of variables. 
For the sake of simplicity, write $\cH$ for $\cH(p, q)$, $\cH_d$ for $\{\lambda\in \cH : |\lambda|\leq d\}$. Throughout the section, we set $ 1\leq i \leq p$, $1\leq j \leq q$.

We now specify using our restricted root system $\Sigma = \Sigma(\gk, \ak)$. We also refer to \cite[Appendix~A]{Z22} for a detailed explanation. 
In particular, we set $m := q, n := p$, $e_j := \dar_j$, $d_i := \epr_i$, $z_i := x_i$, $w_j := y_j$.
By Table~\ref{tab:5para}, $\fk = -1$, $\fq = \fs = -\frac{1}{2}$ while $\fp$ and $\fr$ are both 0. 
By comparing Eq.~(\ref{eqn:SVrho}) with $(-\frac{1}{2}) \rho$ as in Eq.~(\ref{eqn:akRho}), we have $\fh = q-p+\frac{1}{2}$.
Thus we have the two parameters
\[
\fk = -1, \quad \fh = q-p+\frac{1}{2}
\]
for $\Lambda^{(-1, q-p+\frac{1}{2})}_{q,p}$. 
Note our form Eq.~(\ref{eqn:akForm}) is the negative of what is defined in Subsection~\ref{subsec:SVBC}. But this does not matter when $\fk = -1$.
We set (c.f. Eq.~(\ref{eqn:lNat}))
\[
\overline{\lambda^\natural} :=  \sum \lambda_i\epr_i  +\sum\langle \lambda_j'-m \rangle \dar_j.
\]

For convenience we write $\RingEv$ for $\RingEv(\ak^*)$, and define $\RingEv_d := \{f\in \RingEv: \deg f\leq 2d\}$. We define $2(\cH_d)^\natural+\rho := \left\{2\overline{\lambda^\natural}+\rho : \lambda\in \cH_d\right \} \subseteq \ak^*$.

\begin{prop} \label{prop:noShift}
    For each $\mu \in \cH$, there is a unique polynomial $J_\mu \in \RingEv$ of degree $2|\mu|$ such that
    \[
    J_\mu(2\overline{\lambda^\natural}+\rho) = 0,  \quad  \text{for all } \lambda \nsupseteq \mu,\, \lambda \in \cH
    \]
    and that $J_\mu(2\overline{\mu^\natural}+\rho) =  C_\mu^-(1; -1)C_\mu^+(2q-2p; -1)$.
    Moreover, $\{J_\mu:\mu\in \cH_d\}$ is a basis for $\RingEv_d$.
\end{prop}

\begin{proof}
    Let $\tau: z_i \mapsto \frac{1}{2}(x_i-\rho_{2i})$ and $w_j \mapsto \frac{1}{2}(y_j-\rho_{1j})$ be the change of variables which is an isomorphism from $\Lambda^{(-1, q-p+\frac{1}{2})}_{p, q}$ to $\RingEv(\ak^*)$. 
    Consider $I^{\textup{SV}}_{\mu'}(X; -1,  q-p+\frac{1}{2}) \in \Lambda^{(-1, q-p+\frac{1}{2})}_{p, q}$.
    Since $(z(\lambda), w(\lambda)) = (\lambda')^\natural$, by Theorem~\ref{cite:BCPoly}, we see that $J_\mu (X):=\tau(I^{\textup{SV}}_{\mu'}(X; -1,  q-p+\frac{1}{2}))$ satisfies the vanishing properties
    \begin{equation} \label{eqn:fullvan}
        J_\mu(2\overline{\lambda^\natural}+\rho) = 0,  \quad \text{for all } \lambda \nsupseteq \mu,\, \lambda \in \cH.
    \end{equation}
    The normalization condition is straightforward. 
    As $\{I^{\textup{SV}}_{\mu'}: \mu\in \cH_d\}$ is a basis for $\{f\in \Lambda^{(-1, q-p+\frac{1}{2})}_{p, q}: \deg f \leq 2d\}$, $\{J_\mu: \mu\in \cH_d\}$ becomes a basis for $\RingEv_d$. 
\end{proof}

\begin{prop} \label{prop:resSurj}
    Every $f\in \RingEv_d$ is determined by its values on $2(\cH_d)^\natural+\rho$.
\end{prop}
\begin{proof}
    Let $\mathcal{V}_d$ be the space of functions on $2(\cH_d)^\natural+\rho$.
    Then we see that $\dim \RingEv_d = \dim \mathcal{V}_d= |\cH_d|$.
    In particular, $\mathcal{V}_d$ has a Kronecker-delta basis $\{\delta_{\lambda}: 
    \delta_\lambda(2\overline{\lambda^\natural}+\rho) = 1, \delta_\lambda(2\overline{\mu^\natural}+\rho)=0, \lambda, \mu \in \cH_d \}$. 
    Next, the evaluation of $f\in \RingEv_d$ on $\cH_d$ gives a restriction map
    \[
    \mathtt{res}: \RingEv_d \rightarrow \mathcal{V}_d.
    \]
    To prove the statement, we show that $\mathtt{res}$ is an isomorphism. 

    Fix a total order $\succ$ on $\cH_d$ such that $\mu \succ \lambda$ implies $|\mu| \geq |\lambda|$. Consider the matrix $R$ for $\mathtt{res}$ with respect to the bases $\{J_\mu\}$ for $\RingEv_d$ and $\{\delta_\lambda\}$ for $\mathcal{V}_d$ arranged by $\succ$. 
    Since $J_\mu(2\overline{\mu^\natural}+\rho) \neq 0$, and $J_\mu(2\overline{\lambda^\natural}+\rho) = 0$ for any $\lambda$ such that $\mu \succ \lambda$, we see that $R$ is upper triangular with non-zero diagonal entries. 
    Therefore $R$ is invertible, proving the statement.
\end{proof}

To prove Theorem~\ref{thm:A}, we first show that $\hcIsoGK(\sD_\mu)$ is proportional to $J_\mu$ defined above.
Then we pin down the scalar
\begin{equation} \label{eqn:kmu}
    k_\mu := (-1)^{|\mu|}C_\mu^-(1; -1).
\end{equation}
(see Subsection~\ref{subsec:SVBC} for $C_\mu^-$) by comparing the top homogeneous degrees of both sides. 
Here
\[
C_\mu^-(1;-1) = \prod_{(i, j)\in \mu} \left(\mu_i-j+\mu'_j-i+1\right)
\]
(c.f. Subsection~\ref{subsec:SVBC}). The factor of $(-1)^{|\mu|}$ is due to the definition of $d_\mu$, which is a result of the definition of $\varphi_T$ in \cite[Eq.~(41)]{SV05SuperJack}.
\begin{proof}[Proof of Theorem~\ref{thm:A}]
    We show that 
    \[
    \hcIsoGK(\sD_\mu) = k_\mu J_\mu.
    \]
    
    \textit{Step 1.} 
    From the definition of $\hcHomoGK$ and $\hcIsoGK$, we see that $\deg \hcIsoGK(\sD_\mu) \leq 2|\mu|$.
    By the commutative diagram (\ref{diag:tight}), we have $\hcIsoGK(\pi(Z_\mu))=  \Res(\hcZIso(Z_\mu))$.
    The left side is just $\hcIsoGK(\sD_\mu)$.
    The right side, when applied to $2\overline{\lambda^\natural}+\rho$ is just $\hcZIso(Z_\mu) (2\lambda^\natural_\hk+\rho_\hk)$
    by Eq.~(\ref{eqn:resEval}). Hence,
    \begin{equation} \label{eqn:vanProp}
    \hcIsoGK(\sD_\mu)(2\overline{\lambda^\natural}+\rho) = \hcZIso(Z_\mu) (2\lambda^\natural_\hk+\rho_\hk).
    \end{equation}
    By Theorem~\ref{thm:C}, $Z_\mu$ acts by 0 on $I_\lambda$ for all $\lambda \nsupseteq \mu$, $\lambda \in \cH$, thus $Z_\mu$ acts on the quotient $V_\lambda$ by 0. But by Eq.~(\ref{eqn:naturalV}), its $\bk$-highest weight is $2\lambda^\natural_\hk$, and this implies $\hcZIso(Z_\mu)(2\lambda^\natural_\hk+\rho_\tk) = 0$. Therefore
    \[
    \hcIsoGK(\sD_\mu)(2\overline{\lambda^\natural}+\rho) = 0
    \]
    for all $\lambda \nsupseteq \mu$, $\lambda \in \cH$, proving the vanishing properties.

    Alternatively, by Proposition~\ref{prop:weakC}, we have the ``relaxed" vanishing properties $\hcIsoGK(\sD_\mu)(2\overline{\lambda^\natural}+\rho) = 0$ when$ \lambda \neq \mu$, $ |\lambda| \leq |\mu|$.
    By Proposition~\ref{prop:resSurj}, this also implies the full ``extra" vanishing properties.

    Therefore, $\hcIsoGK(\sD_\mu)$ is proportional to $J_\mu$ by its vanishing properties.

    \textit{Step 2.} 
    We first pick a basis for $\pplus$ extended from
    \[
    \{\ii A_{i, p+i}, \ii D_{j, q+j}\},
    \]
    and a dual basis for $\pminus \cong (\pplus)^*$ (identified via the bilinear form $b$) extended from
    \[
    \{-2\ii A_{i, p+i}, 2\ii D_{j, q+j}\}.
    \]
    Thus the identity map $1_{\pplus} \in \End_{\kk}(\pplus)$ corresponds to 
    \begin{equation} \label{eqn:idpplus}
        \sum_{i=1}^p \ii A_{i, p+i}\otimes(-2\ii A_{i, p+i}) + \sum_{j=1}^q\ii D_{j, p+j}\otimes(2\ii D_{j, p+j}) + \text{other terms}
    \end{equation}
    in $\pminus \otimes \pplus$.
    On the other hand, we may extend the basis $\{x_i, y_j\}$ for $\ak$ (see Subsection~\ref{subsec:realize}) to a basis for $\gk$ according to $\gk = \nk^-\oplus \ak \oplus \kk$. This is used to define $\hcHomoGK$ and $\hcIsoGK$.
    Then it is a direct computation to see that the top homogeneous degree of
    \[
    \hcHomoGK(D_{(1)}) = \hcIsoGK(\sD_{(1)})
    \]
    is precisely 
    \[
    \frac{1}{2}\sum_{i=1}^p x_i^2- \frac{1}{2} \sum_{j=1}^q y_j^2.
    \]
    Note the $\rho$ shift does not change this top degree.
    
    Now we let $\mu\in \cH^m$. The identity map $1_\mu$ on $W_\mu$ corresponds to an element in $W_\mu^* \otimes W_\mu$.
    Thus the sum $\sum_{\mu \in \cH^m} 1_{\mu}$ corresponds to an element in $\Sk^m(\pminus) \otimes \Sk^m(\pplus)$ by the decomposition in Proposition~\ref{prop:goodCW}.
    On $\Sk^m(\pk)\subseteq T^m(\pk)$, the bilinear form is induced from $b$ on $\pk$ by normalizing $b^{\otimes m}$ by $1/m!$.
    Similarly, we see that the top homogeneous degree of $\hcIsoGK\left(\sum_{\mu \in \cH^m} \sD_\mu\right)$ is given by
    \begin{equation} \label{eqn:topDeg}
        \frac{1}{m!}\left(\frac{1}{2}\sum_{i=1}^p x_i^2- \frac{1}{2} \sum_{j=1}^q y_j^2 \right)^m =\frac{1}{2^m m!} \left(p_2^{(p, q)}(x_i, y_j)\right)^{m}.
    \end{equation}
    Note $d_\lambda$ (Eq.~(\ref{eqn:dlam})) specializes to $(-1)^{|\lambda|}$.
    From Eq.~(\ref{eqn:SVhat}), Theorems~\ref{cite:SVtab} and \ref{cite:SVThm11}, we see that 
    \[
    I^{\textup{SV}}_{\mu'}\left(x_i, y_j; -1, q-p+\frac{1}{2}\right) = (-1)^{|\mu|}\hat{I}_{\mu}\left( x_i, y_j; -1, \frac{1}{2}-(q-p)\right).
    \]
    and the top homogeneous degree of the right side is $(-1)^{|\mu|}SP_{\mu}(x_i^2, y_j^2; 1)$.
    For $J_\mu$, this becomes
    \[
    (-1/2)^{|\mu|}SP_{\mu}(x_i^2, y_j^2; 1).
    \]
    Thus by setting $\theta = 1$ in Eq.~(\ref{eqn:expJack}), 
    \begin{equation}
        \frac{1}{2^m m!}\left(\sum_{i=1}^p x_i^2- \sum_{j=1}^q y_j^2 \right)^m = \sum_{\mu \in \cH^m} 2^{-m}C_\mu^-(1; -1) (-2)^{|\mu|}\left((-1/2)^{|\mu|}SP_\mu(x_i^2, y_j^2; 1)\right).
    \end{equation}
    Comparing the coefficients, we have
    \begin{equation} 
        \hcIsoGK(\sD_\mu) = (-1)^{|\mu|}C_\mu^-(1; -1) J_\mu.
    \end{equation}
    Therefore the coefficient is precisely $k_\mu$ given by Eq.~(\ref{eqn:kmu}).
\end{proof}

Additionally, we can answer the interesting question of by which scalar $D_\mu$ acts on the spherical vector $\omega \in I_\mu^\kk$, without any direct consideration of the module $I_\mu$ itself! 

\begin{cor} \label{cor:eval}
    The operator $D_\mu \in \Uk^\kk$ acts on $I_\mu^\kk$ by $(-1)^{|\mu|}C_\mu^-(1; -1)^2C_{\mu}^+(2q-2p; -1)$.
\end{cor}
\begin{proof}
This scalar is the value of $\hcHomoGK(D_\mu) = \hcIsoGK(\sD_\mu)$ at $2\overline{\mu^\natural}+\rho$ which equals the value of $k_\mu I^{\textup{SV}}_{\mu'}$ at $(z(\mu'), w(\mu'))$:
    \begin{align*}
    \hcIsoGK(\sD_\mu)(2\overline{\mu^\natural}+\rho) &= k_\mu C_{\mu'}^-(1; -1)C_{\mu'}^+(2q-2p; -1) \\
    &= (-1)^{|\mu|}C_\mu^-(1; -1)^2C_{\mu}^+(2q-2p; -1).  
\end{align*}
\end{proof}

\appendix

\section{Spherical and cospherical modules} \label{app:A}

Let $\gk = \nk \oplus \ak \oplus \kk$ be the Iwasawa decomposition where we set $\nk$ as in Subsection~\ref{subsec:special}. In particular, $\nk\oplus\ak$ extends to the opposite Borel subalgebra $\bk^-$ which is useful in the following consideration.
Let $V = V(\lambda, \bk)$ be irreducible. The dual module $V^*$ is generated by a lowest functional $f_{-}$ of weight $-\lambda$. Then we have $\bk^-.f_- = 0$. We also have $\gk = \kk+\bk^-$ from the Iwasawa decomposition.
\begin{prop}\label{prop:non0} 
In the above notation, if $V$ is spherical and $v^\kk\in V^\kk$ is non-zero, then the canonical pairing $\langle v^\kk, f_{-} \rangle$ is nonzero.
\end{prop}

\begin{proof}
We prove this by way of contradiction. Suppose $\langle v^\kk, f_{-} \rangle = 0$. Consider $\langle v^\kk, u.f_{-} \rangle$ for any $u\in \Uk$. By the Poincar\'e--Birkhoff--Witt theorem, we may write $u =  u_\kk u_{-} $ with $u_\kk \in \Uk(\kk), u_{-} \in \Uk(\bk^-)$. 
Since $\Uk(\bk^-)$ acts on $f_{-}$ by a character and so does $\Uk(\kk)$ on $v^\kk$, we get
\begin{equation*}
    \langle v^\kk, u.f_{-} \rangle = \langle v^\kk, u_\kk u_\qk .f_{-} \rangle
    = b\langle v^\kk, u_\kk .f_{-} \rangle
    =ab\langle v^\kk, f_{-}\rangle
\end{equation*}
where $a, b$ are scalars determined by $u_-$ and $u_\kk$ respectively. The third identity follows by the contragredient action of $u_\kk$ on $v^\kk$.
Thus, $v^\kk$ is a vector on which $\Uk.f_{-} = V^*$ vanishes, which implies $v^\kk=0$, a contradiction.
\end{proof}

\begin{proof}[Proof of Proposition~\ref{cite:scalarClaim}]
We calculate how an arbitrary $D\in \mathfrak{U^\kk}$ acts on $v^\kk$.  Let us consider the pairing $\langle D.v^\kk, f_{-} \rangle$.
By Lemma~\ref{lem:ukScalar}, this pairing equals $c\langle v^\kk, f_{-} \rangle$ for some scalar $c$ depending on $D$.

We write $D = uK+Nu'+p \in \Uk^\kk$ where $p= p(D)$ is the $\Sk(\ak)$ projection of $D$, and $K\in \kk, N\in \nk, u, u'\in \Uk$. By this decomposition, 
\begin{equation*}
    \langle D.v^\kk, f_{-} \rangle = \langle uK.v^\kk, f_{-} \rangle +\langle N u'.v^\kk,  f_{-} \rangle + \langle p.v^\kk,  f_{-} \rangle  =0+0+\langle p.v^\kk, f_{-} \rangle .
\end{equation*}
The second equality follows since $K.v^\kk = 0$ and $N\in \bk^-$ acts by 0 on $f_-$ contragrediently.

We then evaluate the contragredient action of $p$ on $f_-$. Let $p = \sum \prod a_i^{n_i}$ be in $\Sk(\ak) \cong \Pk(\ak^*)$. All $a_i$'s are even and each produces a negative sign when moved to the right side of the pairing, so 
\begin{align*}
    \langle p.v^\kk, f_{-} \rangle    &= \sum \prod \langle v^\kk, (-1)^{n_i}a_i^{n_i}.f_{-} \rangle \\
    &= \sum\prod (-1)^{n_i} \left(a_i\left(-\lambda\right)\right)^{n_i}\langle v^\kk, f_{-} \rangle \\
    &= \sum \prod (-1)^{n_i}(-1)^{n_i}a_i\left(\lambda\right)\langle v^\kk, f_{-} \rangle\\
    &= p\left(\lambda\right)\langle v^\kk, f_{-} \rangle.
\end{align*}
By Proposition~\ref{prop:non0}, we have $c = p\left(\lambda\right)$, which by Eq.~(\ref{eqn:g0+rho=p}) equals $\hcHomoGK(D) \left(\lambda+\rho\right)$.
\end{proof}

Recall $\hk = \ak\oplus\tk_+$ where $\tk_+ = \hk\cap \kk$.
Any $\lambda \in \ak^*$ may be extended to $\hk^*$ by this decomposition, denoted again by $\lambda$, which vanishes on $\tk_+$.
A module $V$ is said to be \emph{cospherical} if $V^*$ is spherical. 
Recall $\bk$ from Subsection~\ref{subsec:special}. We say a $\bk$-highest weight $\lambda$ is \emph{weakly cospherical} if there exists a finite dimensional cospherical module of highest weight $\lambda$. 
In $\ak^*$ we define the following subsets:
\begin{align*}
    F &:= \{\lambda\in \ak^*: V(\lambda, \bk) \text{ is finite dimensional} \}, \\
    C &:= \{\lambda\in \ak^*: \lambda \text{ is weakly cospherical}  \}, \\
    S &:= \{\lambda\in \ak^*: \lambda \text{ is spherical}  \}.
\end{align*}
Clearly $C\subseteq F$. We refer to Definition~\ref{def:sphWt} for $S$.
By \cite[Proposition~4.5]{Z22}, if $V(\lambda, \bk)$ is spherical, then $V(\lambda, \bk)$ is also cospherical. Therefore we have
\[
S\subseteq C \subseteq F.
\]
We show in Proposition~\ref{prop:sphZariski} below that all three subsets are Zariski dense in $\ak^*$.

The following result is \cite[Proposition~6.4.2]{ShermanThesis}. 

\begin{prop} \label{prop:2lambda}
    If $\lambda \in F$, then there exists a finite dimensional module $W$ of $\bk$-highest weight $2\lambda$ which is cospherical. In other words, $2F \subseteq C$.
\end{prop}

\begin{proof}
    Let $\lambda\in F$ and $V := V(\lambda, \bk)$. Let $v\in V$ be a highest weight vector. Denote the representation map as $\phi$. Define a new $\gk$-module structure $V^\theta$ on $V$ by setting $\phi^\theta(X).u := \phi(\theta(X)).u$ for any $X\in \gk$ and $u\in V$.  Notice that $V^\theta $ and $V$ are isomorphic \textit{as $\kk$-modules} since $\theta$ fixes $\kk$. 

    Since $\lambda|_{\tk_+} = 0$, we have $\theta \lambda = -\lambda$. Let $\supp{V}$ denote the set of weights of a weight module $V$.  Thus $\supp{(V^\theta)^*} = \{-\theta(\mu) : \mu \in \supp{V}\} = \supp{V}$. Therefore $(V^\theta)^* \cong V$.
    Let $W \subseteq V \otimes V$ be generated by $w:=v\otimes v$. Then $W$ is finite dimensional and of highest weight $2\lambda$. We show that $W$ is cospherical. 
    Notice that $(V \otimes V)^* \cong (V \otimes (V^\theta)^*)^* \cong \Hom_\gk(V, V^\theta)$. Then the cospherical $\kappa$ on $V \otimes V$, corresponding to $\mathrm{Id}: V \rightarrow V^\theta$ does not vanish on $w$. Thus the restriction of $\kappa$ on $W$ is non-zero. 
\end{proof}

\begin{prop} \label{prop:sphZariski}
    All three subsets $S \subseteq C \subseteq F$ are Zariski dense in $\ak^*$. 
\end{prop} 
\begin{proof}
    We first show that $F$ is Zariski dense. Recall $\lambda^\natural_\hk$ from Eq.~(\ref{eqn:akEvenWt}) and $\lambda^\natural_\hk|_{\tk_+}=0$. In $\ak^*$, we let
    \[
    N := \{\lambda_\hk^\natural: \lambda\in \cH, \; \lambda_1 \geq \cdots \geq \lambda_p \geq \left \langle \lambda_1'-p \right \rangle \geq \cdots \geq \left \langle \lambda_q'-p \right \rangle \geq 0\}.
    \]
    The above $p+q$ inequalities define an open cone in $\ak^*$. Therefore, $N$ is Zariski dense as a subset of the lattice $\mathbb{Z}^{p+q}$ in $\ak^*$.
    By \cite[Proposition~3.4]{Z22}, $V_\lambda = V(2\lambda^\natural_\hk, \bk)$ is finite dimensional for any such $\lambda^\natural_\hk$. 
    Thus $F \supseteq 2N$ is Zariski dense. 

    Since $2F\subseteq C$ by Proposition~\ref{prop:2lambda}, we see that $C$ is Zariski dense.
    
    Finally we show that $S$ is Zariski dense in three steps. In \textit{Step 1}, we construct a special subset $K\subseteq S$. In \textit{Step 2}, we show that $K$ is ``even-supersymmetric Zariski dense", and in \textit{Step 3}, we prove that this implies that $K$ is Zariski dense, hence so is $S$.

    \textit{Step 1.} Let $\lambda \in C$. By definition there is a finite dimensional cospherical $W$ of highest weight $\lambda$.
    Let $\kappa \in (W^*)^\kk$ be non-zero and $M$ be the sum of all submodules of $W$ on which $\kappa$ vanishes. Let $w\in W$ be a highest weight vector. By \cite[Proposition~4.3]{Z22}, we have $\langle w, \kappa \rangle \neq 0$. So $M$ is proper and $\lambda$ is also the highest weight of the quotient $W':= W/M$. 
    By construction, any submodule of $W'$ is cospherical. Let $U$ be an irreducible submodule of $W'$ which exists as $W'$ is finite dimensional.
    By \cite[Proposition~4.5]{Z22}, since $U$ is cospherical and irreducible, $U$ is also spherical. 
    We denote the highest weight of $U$ by $\lambda^\circ$ and define
    \[
    K := \{\lambda^\circ: \lambda\in  C\}\subseteq S.
    \]

    \textit{Step 2.} We show that if $g\in \RingEv(\ak^*)$ and $g$ vanishes on $K$, then $g$ is identically $0$. 
    By the surjectivity of $\Res$ (Proposition~\ref{prop:Res}) and $\hcZIso$, for any $g\in \RingEv(\ak^*)$ there exists $z\in \Zk$ such that $g = \Res \circ \hcZIso(z)$. 
    Consider the action of $z$ on $U$ of highest weight $\lambda^\circ$. Since $U$ is a submodule of $W'$ of highest weight $\lambda$, we have $\hcZIso(z)(\lambda^\circ+\rho) = \hcZIso(z)(\lambda+\rho)$. Suppose $g(K) = 0$, then this implies
    \[
    g(\lambda^\circ)  = \Res\circ\hcZIso(z)(\lambda +\rho) = 0
    \]
    for all $\lambda\in C$. We have proved that $C$ is Zariski dense in $\ak^*$. Hence $g$ is identically 0.

    \textit{Step 3.} We now show that for any $f \in \Pk(\ak^*) = \C[x_i, y_j]$ that vanishes on $K$, $f$ is identically 0.
    Suppose $f(K) = 0$. Set $f_s:= (x_1+y_1)f$, which is independent of $t$ under the substitution $x_1 =- y_1 = t$.
    Recall the Weyl group $W_0$ is $(\mathscr{S}_p \ltimes (\Z/2\Z)^p)\times (\mathscr{S}_q\ltimes (\Z/2\Z)^q)$. We further define 
    \[
    g:=\prod_{w\in W_0} wf_s.
    \]
    We see that $g\in \RingEv(\ak^*)$ and $g(K) = 0$. Here $wf(x):=f(wx)$. By \textit{Step 2}, we know that $g=0$. Since $\Pk(\ak^*)$ is an integral domain, there must be a term $wf_s$ which is 0. Applying $w^{-1}$ we get $f_s = (x_1+y_1)f = 0$. Therefore, we must have $f = 0$.

    In conclusion, $K$ is Zariski dense and thus $S$ is also Zariski dense. 
\end{proof}

\fixIt{
\begin{rmk}\label{rmk:infDim}
    We give an example where $V_\lambda$ fails to be finite dimensional for $(\gl(p+r|q+s), \gl(p|q) \oplus \gl(r|s))$.  Set $p=q=2$ and $s=r=1$. Then Cheng--Wang decomposition gives $(1,1)$-hooks. Take $\lambda = (1,1)$ for example. The $\mathfrak{t}$-weight of $V_\lambda = V(\lambda_\tk^\natural, \bk^\natural)$ is
    \begin{align*}
        \lambda_{\mathfrak{t}}^\natural &= 1\gamma_1^{\textsc{B}}+1\gamma_1^{\textsc{F}} = 1 \epsilon_1^+ +0\epsilon_2^+ + 1\delta_1^+ + 0 \delta_2^+ + (-1)\delta_1^- + (-1)\epsilon_1^- .
    \end{align*}
    In notation of \cite[Section~4.2]{Z22} this weight can also be written as
    \[
    (\overset{\bullet}{1}, \overset{\bullet}{0} | \overset{\times}{1},\overset{\times}{0},\overset{\times}{-1}|\overset{\bullet}{-1}).
    \]
    To see if $V_\lambda$ is finite dimensional or not, we perform a sequence of odd reflections (see \cite[Section~4.2]{Z22} for details). For the above weight, we push $\overset{\bullet}{-1}$ all the way through the deltas ($\times$) and check if the resulting tuple is dominant. When swapping $\overset{\times}{x}| \overset{\bullet}{y}$, we get $\overset{\bullet}{y}|\overset{\times}{x}$ if $x+y=0$, or $\overset{\bullet}{y+1}|\overset{\times}{x-1}$ otherwise.
    \begin{align*}
        &(\overset{\bullet}{1}, \overset{\bullet}{0} | \overset{\times}{1},\overset{\times}{0},\overset{\times}{-1}|\overset{\bullet}{-1}) 
        \leadsto (\overset{\bullet}{1}, \overset{\bullet}{0} | \overset{\times}{1},\overset{\times}{0}|\overset{\bullet}{0}|\overset{\times}{-2}) 
        \leadsto (\overset{\bullet}{1}, \overset{\bullet}{0} | \overset{\times}{1}|\overset{\bullet}{0} |\overset{\times}{0},\overset{\times}{-2}) 
        \leadsto (\overset{\bullet}{1}, \overset{\bullet}{0} , \overset{\bullet}{1} |\overset{\times}{0},\overset{\times}{0},\overset{\times}{-2})
    \end{align*}
    It is clear that the last weight is not dominant. This is due to the extra 0 introduced when $p>r$. Had we taken, for instance, $p=q=r=s=1$, then the weight would be $(\overset{\bullet}{1} | \overset{\times}{1},\overset{\times}{-1}|\overset{\bullet}{-1})$ and the issue would not arise.
\end{rmk}
}

\bibliographystyle{amsalpha}
\bibliography{ref}

\end{document}